\documentclass[12pt,a4paper]{amsart}
\usepackage{amsmath,amssymb,amscd,amsthm}
\usepackage{graphicx}
\usepackage{enumerate}
\input xy
\xyoption{all}

\usepackage{comment}
\usepackage{tikz}
\usepackage{setspace}

\usetikzlibrary{arrows}

\usepackage{xspace}
\usepackage[colorlinks,final,backref=page,hyperindex]{hyperref}
\newtheorem{theorem}{Theorem}[section]
\newtheorem{example}[theorem]{Example}
\newtheorem{remark}[theorem]{Remark}
\newtheorem{lemma}[theorem]{Lemma}
\newtheorem{proposition}[theorem]{Proposition}
\newtheorem{corollary}[theorem]{Corollary}
\numberwithin{equation}{section}
\newtheorem{definition}[theorem]{Definition}

\newcommand{\C}{\mathbb C}

\newcommand{\Z}{\mathbb{Z}}

\def\mg{\mathfrak{g}}

\def\mm{\mathfrak{m}}

\def\mh{\mathfrak{h}}

\def\sl{\mathfrak{sl}}
\def\gl{\mathfrak{gl}}
\def\e{\epsilon}
\def\ba{\mathbf{a}}
\def\br{\mathbf{r}}
\def\bs{\mathbf{s}}
\def\bm{\mathbf{m}}
\def\b1{\mathbf{1}}
\def\bb{\mathbf{b}}
\def\bo{\mathbf{0}}

\def\bu{\mathbf{u}}

\begin{document}
\title[Cuspidal modules]{Classification of Simple Cuspidal Modules over Nongraded Witt Lie Algebras}

\author{Genqiang Liu,
Xiaoyao Zheng, Yufang Zhao
}

\date{}

\begin{abstract}
For a positive integer \(n\), let
$A_n=\mathbb{C}[t_1^{\pm1},\ldots,t_n^{\pm1},x_1,\ldots,x_n]$
and
$\mathfrak{g}_n=\bigoplus_{i=1}^n A_nd_i$,
where \(d_i=t_i\frac{\partial}{\partial t_i}
+\frac{\partial}{\partial x_i}\).
We first determine when the tensor module \(T(P,V)=P\otimes V\) is
simple, where \(P\) is a simple module over the Weyl type algebra \(D_n\) and \(V\) is a simple
\(\mathfrak{gl}_n\)-module. We then prove a canonical algebra isomorphism $A_n\#U(\mathfrak{g}_n)\cong D_n\otimes U(\mathfrak{m}_{\mathbf{1},\mathbf{0}}\Delta)$, and use it
to show that every simple cuspidal \(\mathfrak{g}_n\)-module is
isomorphic to a simple quotient of some
\(T(A_n(\lambda),V)\), where \(V\) is a finite-dimensional simple
\(\mathfrak{gl}_n\)-module.
\end{abstract}

\maketitle


\vskip 5pt
\noindent
{\em 2020  Math. Subj. Class.:}
17B10, 17B20, 17B65, 17B66, 17B68

\section{Introduction}
Nongraded Lie algebras occur naturally in the theory of vertex and conformal
algebras. From the classification of quadratic conformal algebras associated
with certain Hamiltonian pairs \cite{X3,X4}, Xu constructed several families of
nongraded infinite-dimensional Lie algebras arising from locally finite
derivations of commutative associative algebras \cite{X1,X2}. These Lie algebras include
nongraded analogues of the Cartan-type Lie algebras \(W,S,H\), and \(K\).
The present paper concerns representations over the nongraded Lie algebras of type \(W\).

Representations of graded Witt-type Lie algebras have been studied
extensively. Weight modules with one-dimensional weight
spaces over graded generalized Witt algebras were classified in \cite{ZK}. Second cohomology groups were computed  for generalized
Witt-type Lie algebras in \cite{SZh2}.  Modules analogous to intermediate series  were classified
for nongraded generalized Witt algebras in \cite{SZ1}. Zhao subsequently
constructed irreducible generalized weight modules for nongraded Lie
algebras of type $W$ from finite-dimensional modules over general linear
Lie algebras \cite{Z}. These constructions are nongraded analogues of
Rao's tensor-field modules for derivation Lie algebras $W_n$ of Laurent polynomial
algebras \cite{Ra1,Ra2}.

Further landmarks include Rudakov's work on smooth representations of
Cartan-type Lie algebras over formal power-series algebras \cite{R1,R2},
Mathieu's classification of Harish--Chandra modules over the Virasoro
algebra \cite{M1}, and Penkov and Serganova's determination of the supports
of simple weight modules over \(W_n\) and \(\bar S_n\) \cite{PS}. Shen's
mixed-product construction \cite{Sh} led to the systematic study of tensor
modules over Witt type Lie algebras built from Weyl-algebra modules and \(\mathfrak{gl}_n\)-modules;
see, for example, \cite{LLZ}.
Simple  weight modules over Laurent
vector-field algebras $W_n$ with finite-dimensional weight spaces were classified by  Billig and Futorny, see \cite{BF1}.
Simple uniformly bounded weight modules over polynomial
vector-field algebras $W_n^+$ were classified by Xue and L\"u, see \cite{XL2}. Simple  weight modules over $W_n^+$ with finite-dimensional weight spaces were classified by  Grantcharov and Serganova, see \cite{GS}.

In this paper, we extend Zhao's simplicity criterion for the
tensor modules in \cite{Z} and classify all simple
cuspidal modules over \(\mathfrak g_n\). In Section~2, we introduce generalized weight modules over $\mg_n$,
the relevant
Weyl-type algebra $D_n$ and general tensor modules. In Section~3, we establish a
simplicity criterion for \(T(P,V)\) for any simple $D_n$-module $P$ and any simple $\gl_n$-module $V$, and analyze the exceptional exterior
power modules \(V(\delta_r,r)\), see Theorems~\ref{simplicity}
and~\ref{remain}. In Section~4, we first classify simple cuspidal
\((A_n,\mathfrak g_n)\)-modules (Proposition~\ref{(A_n,g_n)-module}) and
then use the isomorphism Theorem \ref{mainth1} and the \(A\)-cover construction to obtain the classification of
simple cuspidal \(\mathfrak g_n\)-modules, see Theorem~\ref{cuspidal}.

We denote by \(\Z\), \(\Z_{\geq0}\), \(\Z_{\geq1}\), and \(\C\) the
sets of integers, nonnegative integers, positive integers, and complex
numbers, respectively. All vector spaces and algebras are over \(\C\).
For a Lie algebra \(\mathfrak g\), we denote its universal enveloping
algebra by \(U(\mathfrak g)\). We write \(\otimes\) for
\(\otimes_{\mathbb C}\). An associative algebra \(A\) will also be
viewed as a Lie algebra \(A_L\) under the commutator bracket
\([x,y]=xy-yx\).

\section{Preliminaries and main results}

In this section, we prepare some necessary preliminaries, including
the algebra $\mg_n$, cuspidal modules over $\mg_n$, the Weyl type algebra $D_n$, the tensor $\mg_n$-modules $T(P,V)$, and the main results in this paper.

\subsection{The nongraded Witt Lie algebra}

Let \(\{e_1,e_2,\ldots,e_n\}\) be the standard basis of \(\mathbb{C}^n\), and set
\[
A_n=\mathbb{C}[t_1^{\pm1},\ldots,t_n^{\pm1},
x_1,\ldots,x_n].
\]
For convenience, let \(d_i=t_i\frac{\partial}{\partial t_i}
+\frac{\partial}{\partial x_i}\) for \(i=1,\ldots,n\), and write
\(t^{\ba}=t_1^{a_1}\cdots t_n^{a_n}\) and
\(x^{\bs}=x_1^{s_1}\cdots x_n^{s_n}\) for
\(\ba\in\Z^n\) and \(\bs\in\Z_{\geq0}^n\).
Set \(\mg_n=\bigoplus_{i=1}^nA_nd_i\) which is a Lie subalgebra
of $\mathrm{Der}(A_n)$. Its Lie bracket is
$$[t^{\ba}x^{\bm}d_{i},t^{\bb}x^{\bs}d_{j}]=
t^{\ba+\bb}(b_ix^{\bs+\bm}+s_ix^{\bs+\bm-e_i})d_j
-t^{\ba+\bb}(a_jx^{\bs+\bm}+m_jx^{\bs+\bm-e_j})d_i,
$$
where \(\ba,\bb\in\Z^n\), \(\bm,\bs\in\Z_{\geq0}^n\), and
\(1\leq i,j\leq n\). By \cite[Theorem~2.2]{X2}, \(\mg_n\) is a simple Lie algebra.
Let
$$W_n=\oplus_{i=1}^n\mathbb{C}[t_1^{\pm 1},\cdots ,t_n^{\pm 1}]d_i, \quad W_n^+=\oplus_{i=1}^n\mathbb{C}[x_1,\cdots ,x_n]d_i,$$ which are subalgebras of $\mg_n$.
It is easy to see that  $W_n\cong\mathrm{Der}(\mathbb{C}[t_1^{\pm 1},\cdots ,t_n^{\pm 1}])$
and $W_n^+\cong\mathrm{Der}(\mathbb{C}[x_1,\cdots ,x_n])$.
Let $\Delta=\mathrm{span}\{d_1,\dots,d_n\}$ which is a commutative subalgebra of $\mg_n$. In this paper, we will study $\mg_n$-modules that are locally finite over $\Delta$.

\begin{definition}
A \(\mg_n\)-module \(M\) is called a generalized weight module if
\(M=\bigoplus_{\lambda\in\C^n}M(\lambda)\), where
$$M(\lambda)=\{v \in M \mid (d_i-\lambda_i)^N v=0,\  \text{for every}\ i=1,\dots,n,
\text{and some } N>0\}.$$
The subspace \(M(\lambda)\) is called the generalized weight space of
weight \(\lambda\), and
\[
\operatorname{supp}(M)=\{\lambda\in\C^n\mid M(\lambda)\neq0\}
\]
is called the support of \(M\).
\end{definition}

For a generalized weight $\mg_n$-module  $M=\oplus_{\lambda \in \C^{n} }M(\lambda)$,
since \(\operatorname{ad}d_i\) acts locally nilpotently on
$W_n^+$, we have $W_n^+M(\lambda)\subset M(\lambda)$. Hence
each generalized weight space
$M(\lambda)$ is a $W_n^+$-module on which
$d_i-\lambda_i$ acts locally nilpotently, i.e.,
$M(\lambda)$ is a Whittaker \(W_n^+\)-module of type \(\lambda\),
see \cite{ZL}. Since \(W_n^+\) is an infinite-dimensional simple Lie
algebra, every finite-dimensional representation of \(W_n^+\) is trivial.
Thus \(M(\lambda)\) is infinite-dimensional whenever \(M\) is nontrivial.

\begin{definition}For a  generalized weight module $M$ over $\mg_n$, and $\lambda\in \emph{supp}(M)$, we also define
$$M_\lambda=\{v \in M \mid d_iv=\lambda_i v,\  \text{for every}\ i=1,\dots,n.\}.$$
Clearly $M_\lambda\subset M(\lambda)$.
The subspace \(M_\lambda\) is called the weight space of weight
\(\lambda\). A generalized weight \(\mg_n\)-module \(M\) is called
cuspidal if there exists \(r\in\Z_{\geq1}\) such that
\(\dim M_\lambda\leq r\) for every
\(\lambda\in\operatorname{supp}(M)\).
\end{definition}

\begin{remark}
Cuspidal modules are also called uniformly bounded modules.
We use the term cuspidal module because these modules cannot be obtained through the parabolic induction.
\end{remark}

\begin{example}
Let $M=A_n$ be the natural module of $\mg_n$. From
$$[d_{i},t^{\bb}x^{\bs}]=
b_it^{\bb}x^{\bs}+s_it^{\bb}x^{\bs-e_i},
$$
we see that $M=\oplus_{\bb\in \Z^n} M(\bb)$ is a generalized weight module, where
$$M(\bb)=t^{\bb}\C[x_1,\dots,x_n].
$$
Moreover the weight space $M_\bb=\C t^{\bb}$ for every $\bb\in \Z^n$.
\end{example}

\begin{example}
Let $V=\mg_n$ be the adjoint module of $\mg_n$. From
$$[d_{i},t^{\bb}x^{\bs}d_{j}]=
b_it^{\bb}x^{\bs}d_j+s_it^{\bb}x^{\bs-e_i}d_j,
$$
we see that $V=\oplus_{\bb\in \Z^n} V(\bb)$ is a generalized weight module, where
$$V(\bb)=\mathrm{span}
\{t^{\bb}x^{\bs}d_{j}\mid  \bs\in\Z_{\geq 0}^n, j=1,\dots,n\}.
$$
Moreover the weight space $V_\bb=\mathrm{span}
\{t^{\bb}d_{j}\mid  j=1,\dots,n\}$ for every $\bb\in \Z^n$.
Since $\mg_n$ is a simple Lie algebra,  $V$ is a simple cuspidal module over  $\mg_n$.
\end{example}

Let \(\widetilde{\mg}_n=\mg_n\ltimes A_n\), where
\[
[t^{\ba}x^{\bm}d_i,t^{\bb}x^{\bs}]
=t^{\ba}x^{\bm}d_i(t^{\bb}x^{\bs}),\quad [A_n,A_n]=0.
\]
We call \(\widetilde{\mg}_n\) the extended Witt type algebra.

\begin{definition}A $\widetilde{\mg}_n$-module is called an $(A_n,\mg_n)$-module
if the action of $A_n$ on $M$ is associative.
\end{definition}

\begin{example}The adjoint module $\mg_n$ is an $(A_n,\mg_n)$-module
under the natural action of $A_n$:
$$t^{\ba}x^{\bm} \cdot t^{\bb}x^{\bs}d_j=t^{\ba+\bb}x^{\bs+\bm}d_j.$$
\end{example}

Since  $A_n$ is a left  module algebra over the Hopf algebra $U(\mg_n)$, we have the smash product algebra $A_n\# U(\mg_n)$. An $(A_n,\mg_n)$-module
is exactly a module over  $A_n\# U(\mg_n)$. Let $D_n$ be the associative subalgebra of $A_n\#U(\mg_n)$ generated by $A_n$ and $d_1,\dots,d_n$. We call it a Weyl-type algebra, and it satisfies the following relations:
$$\aligned \
[d_i,d_j]=0, [t^{\ba}x^{\bm},t^{\bb}x^{\bs}]=0, [d_i,t^{\ba}x^{\bm}]=a_it^{\ba}x^{\bm}+m_it^{\ba}x^{\bm-e_i},
\endaligned $$
for all $1\leq i,j\leq n$, $\ba,\bb\in \Z^n$, $\bm,\bs\in \Z^n_{\geq0}$. We will show that $D_n$ is a tensor product factor
of $A_n\# U(\mg_n)$, see Theorem \ref{mainth}.

\subsection{Tensor  modules over $\mg_n$}

Let $E_{i,j}$ be the $n\times n$ matrix whose $(i,j)$-entry is $1$ and all other entries are $0$. For any $D_n$-module $P$ and  $\gl_n$-module $V$, set $T(P, V)=P\otimes V$. By
Theorem~\ref{mainth1} and Lemma~\ref{lem:gln-quotient}, we obtain the
following Lie algebra homomorphism:
$$\aligned \
\sigma: \mg_n &\rightarrow D_n\otimes U(\gl_n),\\
t^{\bb}x^{\bs}d_j &\mapsto t^{\bb}x^{\bs}d_j\otimes 1 + \sum^n_{i=1}d_i(t^{\bb}x^{\bs})\otimes E_{i,j},
\endaligned$$
for all $\bb\in \Z^n$, $\bs\in \Z^n_{\geq0}$.
Then \(T(P,V)\) has an \((A_n,\mg_n)\)-module
structure such that
$$\aligned
t^{\bb}x^{\bs}d_j(p\otimes v)&=(t^{\bb}x^{\bs}d_j p)\otimes v
+\sum_{i=1}^n(b_it^{\bb}x^{\bs}p+s_it^{\bb}x^{\bs-e_i}p)\otimes E_{i,j}v,\\
t^{\ba}x^{\bm}(p\otimes v)&=(t^{\ba}x^{\bm}p)\otimes v, \ \ p\in A_n, v\in V, \ba,\bb\in \Z^n, \bm,\bs\in \Z^n_{\geq0}.
\endaligned$$

When  $P=t_1^{\lambda_1}\dots t_n^{\lambda_n} A_n$, $\lambda\in \C^n$, and $V$ are finite-dimensional simple $\gl_n$-modules, simplicities of these $\mg_n$-modules $T(P, V)$  were given  in \cite{Z}. We will decide
when \(T(P,V)\) is simple for arbitrary simple \(D_n\)-modules \(P\)
and simple \(\gl_n\)-modules \(V\).


We know that
$$\gl_n=\bigoplus_{1\leq i,j\leq n}\C E_{i,j}, \quad \sl_n=\C\mh \oplus (\bigoplus_{1\leq i\neq j\leq n}\C E_{i,j}),$$
where $\mh=\text{span}\{h_i=E_{i,i}-E_{i+1,i+1}\mid i=1,\dots, n-1\}$. Let
$$\Lambda^{+}=\{\lambda\in \mh^* \mid \lambda(h_i)\in \Z_{\geq0}, \forall\ i=1,\dots, n-1\}$$
be the set of dominant integral weights with respect to \(\mh\). For
\(\psi\in\Lambda^+\), let \(V(\psi)\) be the simple
\(\sl_n\)-module of highest weight \(\psi\). Extend
\(V(\psi)\) to be a
\(\gl_n\)-module \(V(\psi,b)\) by letting the identity matrix $I$ act as
the scalar \(b\). For \(0\leq i\leq n\), define
\(\delta_i\in\mh^*\) by
\[
\delta_i(h_j)=\delta_{ij},\qquad 1\leq j\leq n-1,
\]
where \(\delta_0=\delta_n=0\).

Consider the exterior product
\begin{align*}
\bigwedge^{r}(\C^n)=\C^n\wedge \cdots \wedge \C^n \quad (r \, \text{times}),
\end{align*}
which is a $\gl_n$-module by
$$X(v_1\wedge\cdots \wedge v_r)=\sum^r_{i=1}v_1\wedge \cdots \wedge Xv_i\wedge \cdots \wedge v_r, \forall\ X\in \gl_n.$$
And there is a $\gl_n$-module isomorphism $V(\delta_r,r)\cong \bigwedge^{r}(\C^n)$. In particular, $V(\delta_0,0)$ and $V(\delta_n,n)$ are one-dimensional $\gl_n$-module.

\subsection{Main results}

In this paper, we will show the following two main results.

\begin{theorem}
Let $P$ be a simple $D_n$-module,  $V$  a simple $\gl_n$-module that is not isomorphic to $V(\delta_r,r)$ for any $r\in \{0, 1,\cdots, n\}$. Then $T(P, V)$ is a simple module over $\mg_n$.
\end{theorem}

Simple sub-quotients of $T(P, V(\delta_r,r))$ are also
determined in this paper.

\begin{theorem}\label{thm:main-classification}
If \(M\) is a simple cuspidal \(\mg_n\)-module, then \(M\) is
isomorphic to a simple quotient of \(T(A_n(\lambda),V)\) for some
\(\lambda\in\C^n\) and some finite-dimensional simple
\(\gl_n\)-module \(V\).
\end{theorem}

We first show that any simple cuspidal \((A_n,\mg_n)\)-module is
isomorphic to some \(T(A_n(\lambda),V)\). We then use the \(A\)-cover
construction to classify
simple cuspidal module over  $\mg_n$.

\section{Simplicity of $T(P,V)$}

In this section, we determine when \(T(P,V)\) is simple. We first
prove Theorem~\ref{simplicity}. We then discuss the exceptional case
\(V\cong V(\delta_r,r)\), \(0\leq r\leq n\), and describe all simple  subquotients of \(T(P,V(\delta_r,r))\).

\subsection{Simplicity of $T(P,V)$ when $V\ncong  V(\delta_{r},r)$ }

In this subsection, we address the case that $V\ncong  V(\delta_{r},r)$.

\begin{theorem}\label{simplicity}
Let \(P\) be a simple \(D_n\)-module and \(V\) a simple
\(\gl_n\)-module. If \(T(P,V)\) is reducible as a \(\mg_n\)-module,
then \(V\cong V(\delta_r,r)\) for some \(r\in\{0,1,\ldots,n\}\).
\end{theorem}

\begin{proof}
Suppose that \(T(P,V)\) is reducible. Let \(T'\) be a nonzero proper
\(\mg_n\)-submodule of \(T(P,V)\), and choose a nonzero element
\(\sum_{k=1}^{q}p_k\otimes v_k\in T'\). We prove three claims.

\smallskip
\noindent\textbf{Claim 1.}
For every \(u\in D_n\) and \(1\leq i,j,l\leq n\), one has
$$\sum^{q}_{k=1}up_k\otimes (\delta_{l,i}E_{l,j}-E_{l,i}E_{l,j})v_k\in T'.$$

For any $1\leq i,j,l \leq n$, $m=0,1,2$, $\bb=(b_1,\dots, b_n)\in \Z^n$, $\bs=(s_1,\dots, s_n)\in \Z_{\geq0}^n$ with $s_l\geq2$. Let $\Gamma_m=x^{me_l}d_i t^{\bb}x^{\bs-me_l}d_j\in U(\mg_n)$. Then we have
$$\aligned \
&\Gamma_m(\sum^{q}_{k=1}p_k\otimes v_k)\\
=&\sum^{q}_{k=1}(x^{me_l}d_i)(t^{\bb}x^{\bs-me_l}d_j p_k\otimes v_k)\\
 &+ \sum^{q}_{k=1}\sum^n_{c=1}(x^{me_l}d_i)(b_ct^{\bb}x^{\bs-me_l}+(s_c-m\delta_{c,l})t^{\bb}x^{\bs-me_l-e_c})p_k \otimes E_{c,j}v_k\\
=&\sum^q_{k=1}((b_it^{\bb}x^{\bs}+(s_i-m\delta_{i,l})t^{\bb}x^{\bs-e_i})d_jp_k\otimes v_k \\
&+ t^{\bb}x^{\bs}d_id_jp_k\otimes v_k+ mt^{\bb}x^{\bs-\e_l}d_jp_k\otimes E_{l,i}v_k)\\
&+ \sum^q_{k=1}\sum^n_{c=1}b_c((b_it^{\bb}x^{\bs}+(s_i-m\delta_{i,l})t^{\bb}x^{\bs-e_i})p_k\otimes E_{c,j}v_k\\
&+ t^{\bb}x^{\bs}d_ip_k\otimes E_{c,j}v_k + mt^{\bb}x^{\bs-e_l}p_k\otimes E_{l,i}E_{c,j}v_k)\\
&+ \sum^q_{k=1}\sum^n_{c=1}(s_c-m\delta_{c,l})((b_it^{\bb}x^{\bs-e_l}+(s_i-m\delta_{i,l}-\delta_{i,c})t^{\bb}x^{\bs-e_c-e_i})p_k\otimes E_{c,j}v_k\\
&+ t^{\bb}x^{\bs-e_c}d_ip_k\otimes E_{c,j}v_k + mt^{\bb}x^{\bs-e_l-e_c}p_k\otimes E_{l,i}E_{c,j}v_k)\\
=&m^2\sum^q_{k=1}t^{\bb}x^{\bs-2e_l}p_k\otimes
(\delta_{l,i}E_{l,j}-E_{l,i}E_{l,j})v_k+mu_1+u_0,
\endaligned$$
where $u_0$,$u_1$ are decided by $i,j,l, \bb, \bs$ and do not depend on $m$. Take $m=0,1,2$, we have
\begin{equation}\label{Gamma1}
\frac{1}{2}(\Gamma_2-2\Gamma_1+\Gamma_0)\sum^q_{k=1}p_k\otimes v_k=\sum^q_{k=1}t^{\bb}x^{\bs'}p_k\otimes (\delta_{l,i}E_{l,j}-E_{l,i}E_{l,j})v_k\in T',
\end{equation}
for all $\bb\in \Z^n$, $\bs'\in \Z_{\geq0}^n$, $1\leq i,j,l \leq n$.

For all $1\leq r \leq n$, applying $d_r$ to (\ref{Gamma1}), we get
\begin{equation}\label{Gamma2}
\sum^q_{k=1}d_rt^{\bb}x^{\bs'}p_k\otimes (\delta_{l,i}E_{l,j}-E_{l,i}E_{l,j})v_k\in T', \forall \ \bb\in \Z^n, \bs'\in \Z_{\geq0}^n.
\end{equation}
Since $D_n$ is generated by $A_n$ and $d_1,\dots,d_n$, we have
$$\sum^{q}_{k=1}up_k\otimes (\delta_{l,i}E_{l,j}-E_{l,i}E_{l,j})v_k\in T', \forall \ u\in D_n.$$
\smallskip
\noindent\textbf{Claim 2.}
If \(p_1,\ldots,p_q\) are linearly independent, then
$$(\delta_{l,i}E_{l,j}-E_{l,i}E_{l,j})v_k=0.$$

The algebra \(D_n\) has countable dimension over the uncountable
algebraically closed field \(\C\). Hence
\(\operatorname{End}_{D_n}(P)=\C\) by Dixmier's lemma
\cite{Dixmier}. The Jacobson
density theorem therefore implies that, for every \(p\in P\), there
exists \(u(k,p)\in D_n\) such that
\[
u(k,p)p_i=\delta_{ki}p,\qquad 1\leq i\leq q.
\]
Claim~1 now gives
\[
P\otimes(\delta_{l,i}E_{l,j}-E_{l,i}E_{l,j})v_k\subseteq T'.
\]
Set
\[
T_1=\{v\in V\mid P\otimes v\subseteq T'\}.
\]
For \(v\in T_1\) and \(p\in P\), we have
$$(x_jd_i)(p\otimes v)=(x_jd_ip)\otimes v+p\otimes E_{j,i}v\in T', \forall 1\leq i,j\leq n.$$
Thus \(E_{j,i}v\in T_1\), so \(T_1\) is a \(\gl_n\)-submodule of
\(V\). Since \(V\) is simple, \(T_1=0\) or \(V\). Because \(T'\) is
proper, \(T_1\neq V\); hence \(T_1=0\). Therefore
$$(\delta_{l,i}E_{l,j}-E_{l,i}E_{l,j})v_k=0$$
for all \(1\leq i,j,l\leq n\) and \(1\leq k\leq q\). This proves
Claim~2.

\smallskip
\noindent\textbf{Claim 3.}
For all \(1\leq i,j,l\leq n\),
\[
(\delta_{l,i}E_{l,j}-E_{l,i}E_{l,j})V=0.
\]

For all $1\leq a,b \leq n$, we have
\begin{align*}
(x_ad_b)(\sum^q_{k=1}p_k\otimes v_k)=\sum^q_{k=1}(x_ad_bp_k\otimes v_k + p_k\otimes E_{a,b}v_k)\in T'.
\end{align*}
From Claim 1, we see that
$$\sum^q_{k=1}(ux_ad_bp_k\otimes (\delta_{l,i}E_{l,j}-E_{l,i}E_{l,j})v_k + up_k\otimes (\delta_{l,i}E_{l,j}-E_{l,i}E_{l,j})E_{a,b}v_k)\in T'$$
for all $u\in D_n$. And by Claim 2, we obtain
$$\sum^q_{k=1}up_k\otimes (\delta_{l,i}E_{l,j}-E_{l,i}E_{l,j})E_{a,b}v_k\in T', \forall u\in D_n.$$
By the Jacobson density theorem, we have
$$P\otimes (\delta_{l,i}E_{l,j}-E_{l,i}E_{l,j})E_{a,b}v_k \subseteq T', \forall k=1,\dots,q.$$
Therefore $(\delta_{l,i}E_{l,j}-E_{l,i}E_{l,j})E_{a,b}v_k\in T_1$, i.e.,
$$(\delta_{l,i}E_{l,j}-E_{l,i}E_{l,j})E_{a,b}v_k=0.$$
Repeat the above steps, we have
$$(\delta_{l,i}E_{l,j}-E_{l,i}E_{l,j})U(\gl_n)v_k=0.$$
Since \(V\) is a simple \(\gl_n\)-module, this proves Claim~3.

Observing Claim 3, we get
\begin{equation}\label{locally nil}
(E_{i,i}-E_{i,i}^2)V=E_{l,i}^2V=0, \forall \ 1\leq l\neq i\leq n.
\end{equation}
By \cite[Lemma~2.4(2)]{LZ}, \(V\) is finite-dimensional. Hence \(V\)
is a highest-weight \(\gl_n\)-module. Let \(v_\alpha\) be a
highest-weight vector and write
\[
E_{ii}v_\alpha=\alpha_i v_\alpha.
\]
Equation~\eqref{locally nil} implies that
\(\alpha_i\in\{0,1\}\) for every \(i\). Dominance gives
\(\alpha_1\geq\cdots\geq\alpha_n\), so there is an
\(N\in\{0,1,\ldots,n\}\) such that
\[
\alpha_1=\cdots=\alpha_N=1,
\qquad
\alpha_{N+1}=\cdots=\alpha_n=0.
\]
Consequently,
$V\cong V(\delta_N,N)$,
as required.
\end{proof}

We next give an isomorphism criterion for simple tensor modules.

\begin{proposition}
Let \(P,P'\) be simple \(D_n\)-modules and \(V,V'\) simple
\(\gl_n\)-modules. Assume that
\(V\ncong V(\delta_r,r)\) for \(0\leq r\leq n\). Then
\[
T(P,V)\cong T(P',V')
\]
if and only if \(P\cong P'\) and \(V\cong V'\).
\end{proposition}
\begin{proof}
The sufficiency is obvious.
Assume that
$$\aligned \
f: T(P,V) &\rightarrow T(P',V')\\
p\otimes v &\mapsto \sum^m_{k=1}p_k\otimes v_k,
\endaligned$$
is an isomorphism of \(\mg_n\)-modules, where
\(p_1,\ldots,p_m\) are linearly independent. By Claim~1 in the proof
of Theorem~\ref{simplicity}, we have
\begin{equation}\label{mapf}
f(up\otimes (\delta_{l,i}E_{l,j}-E_{l,i}E_{l,j})v)=\sum^m_{k=1}up_k\otimes (\delta_{l,i}E_{l,j}-E_{l,i}E_{l,j})v_k,
\end{equation}
for every \(u\in D_n\) and \(1\leq i,j,l\leq n\).
Since \(V\ncong V(\delta_r,r)\), Claim~3 in the proof of
Theorem~\ref{simplicity} shows that we may choose \(i,j,l\) and
\(v\in V\) such that
\[
(\delta_{l,i}E_{l,j}-E_{l,i}E_{l,j})v\neq0.
\]
Equation~\eqref{mapf} then implies that the corresponding expression
is nonzero for at least one \(v_k\); relabeling, assume it is nonzero
for \(v_1\). Since \(\operatorname{End}_{D_n}(P')=\C\), the Jacobson
density theorem gives \(u_1\in D_n\) such that
\(u_1p_k=\delta_{1k}p_1\). Therefore
$$f(u_1p\otimes (\delta_{l,i}E_{l,j}-E_{l,i}E_{l,j})v)=u_1p_1\otimes (\delta_{l,i}E_{l,j}-E_{l,i}E_{l,j})v_1\neq0,$$
which implies that \(u_1p\neq0\). Using \eqref{mapf} once more gives
$$f(u u_1p\otimes (\delta_{l,i}E_{l,j}-E_{l,i}E_{l,j})v)=up_1\otimes (\delta_{l,i}E_{l,j}-E_{l,i}E_{l,j})v_1.$$
Denote $u_1p=\bar{p}$, $(\delta_{l,i}E_{l,j}-E_{l,i}E_{l,j})v=\bar{v}$, $p_1=\tilde{p}$, $(\delta_{l,i}E_{l,j}-E_{l,i}E_{l,j})v_1=\tilde{v}$, then
$$f(u\bar{p}\otimes \bar{v})=u\tilde{p}\otimes \tilde{v}, \forall u\in D_n.$$
So $\text{Ann}_{D_n}(\bar{p})=\text{Ann}_{D_n}(\tilde{p})$. Since $P$ and $P'$ are irreducible, $P\cong D_n/\text{Ann}_{D_n}(\bar{p})=D_n/\text{Ann}_{D_n}(\tilde{p})\cong P'$.

We denote the map $f':P \rightarrow P'$ with $f'(u\bar{p})=u\tilde{p}$, clearly this is a $D_n$-module isomorphism. And we have $f(p\otimes \bar{v})=f'(p)\otimes \tilde{v}$. Note that
$$\aligned \
x_id_jf(p\otimes \bar{v})=&x_id_jf'(p)\otimes \tilde{v}+ f'(p)\otimes E_{i,j}\tilde{v}\\
=&f(x_id_j(p\otimes \bar{v}))\\
=&f'(x_id_jp)\otimes \tilde{v}+ f(p\otimes E_{i,j}\bar{v}),
\endaligned$$
then \(f(p\otimes E_{i,j}\bar v)=f'(p)\otimes E_{i,j}\tilde v\)
for every \(p\in P\) and \(1\leq i,j\leq n\). It follows that
\(f(p\otimes\mu\bar v)=f'(p)\otimes\mu\tilde v\) for all
\(\mu\in U(\gl_n)\). Therefore,
\[
\operatorname{Ann}_{U(\gl_n)}(\bar v)
=\operatorname{Ann}_{U(\gl_n)}(\tilde v).
\]
Since \(V\) and \(V'\) are simple, this equality yields
\[
V\cong U(\gl_n)/\operatorname{Ann}_{U(\gl_n)}(\bar v)
\cong U(\gl_n)/\operatorname{Ann}_{U(\gl_n)}(\tilde v)
\cong V'.
\]
\end{proof}

\subsection{Structure of $T(P,V)$ when $V\cong  V(\delta_{r},r)$ }

We now consider the exceptional modules
\(V\cong V(\delta_r,r)\), \(0\leq r\leq n\). More generally,
\(V(\delta_0,b)=\mathbb Cv\) denotes the one-dimensional
\(\gl_n\)-module on which the identity matrix acts as the scalar \(b\).
As a vector space, \(T(P,V(\delta_0,b))\cong P\), with action
$$t^{\ba}x^{\bm}d_j(p)=t^{\ba}x^{\bm}d_jp+\frac{b}{n}(a_jt^{\ba}x^{\bm}+m_jt^{\ba}x^{\bm-e_j})p, \forall\ p\in P.$$

\begin{lemma}\label{homomorphismpi}
Let $P$ be a $D_n$-module, then
$$\aligned
\pi_k: T(P,V(\delta_k,k)) &\rightarrow T(P,V(\delta_{k+1},k+1))\\
p\otimes v &\rightarrow \sum^n_{i=1}d_ip\otimes e_i \wedge v, \forall\ p\in P, v\in V(\delta_k,k)
\endaligned$$
is a $\mg_n$-module homomorphism, where $k=0,1,\dots, n-1$. Furthermore, we have $\pi_{k+1}\pi_{k}=0$.
\end{lemma}

\begin{proof}
It suffices to prove that
$$\pi_k(t^{\bb}x^{\bs}d_j(p\otimes v))=t^{\bb}x^{\bs}d_j\pi_k(p\otimes v), \forall\ \bb\in \Z^n, \bs\in \Z^n_{\geq0}, 1\leq j\leq n.$$
Note that
$$\aligned
\pi_k(t^{\bb}x^{\bs}d_j(p\otimes v))&=\pi_k(t^{\bb}x^{\bs}d_jp\otimes v+\sum^n_{l=1}d_l(t^{\bb}x^{\bs})p\otimes E_{l,j}v)\\
&=\sum^n_{i=1}d_it^{\bb}x^{\bs}d_jp\otimes e_i\wedge v+\sum^n_{i=1}\sum^n_{l=1}d_id_l(t^{\bb}x^{\bs})p\otimes e_i\wedge E_{l,j}v.
\endaligned$$
Using $e_i\wedge E_{l,j}v=-e_l\wedge E_{i,j}v$ for all $1\leq i,l\leq n$, we have
$$\aligned
&t^{\bb}x^{\bs}d_j\pi_k(p\otimes v)\\
=&t^{\bb}x^{\bs}d_j(\sum^n_{i=1}d_ip\otimes e_i \wedge v)\\
=&\sum^n_{i=1}t^{\bb}x^{\bs}d_jd_ip\otimes e_i \wedge v+\sum^n_{i=1}\sum^n_{l=1}d_l(t^{\bb}x^{\bs})d_i p\otimes E_{l,j}(e_i\wedge v)\\
=&\sum^n_{i=1}(d_it^{\bb}x^{\bs}d_jp\otimes e_i \wedge v-d_i(t^{\bb}x^{\bs})d_j p\otimes e_i \wedge v)+\sum^n_{l=1}d_l(t^{\bb}x^{\bs})d_jp\otimes e_l\wedge v\\
&+\sum^n_{i=1}\sum^n_{l=1}d_id_l(t^{\bb}x^{\bs})p\otimes e_i\wedge(E_{l,j}v)-\sum^n_{i=1}\sum^n_{l=1}d_i(d_l(t^{\bb}x^{\bs}))p\otimes e_i\wedge(E_{l,j}v)\\
=&\sum^n_{i=1}d_it^{\bb}x^{\bs}d_jp\otimes e_i\wedge v+\sum^n_{i=1}\sum^n_{l=1}d_id_l(t^{\bb}x^{\bs})p\otimes e_i\wedge E_{l,j}v\\
=&\pi_k(t^{\bb}x^{\bs}d_j(p\otimes v)).
\endaligned$$
It follows immediately that
\(\pi_{k+1}\pi_k(p\otimes v)
=\sum_{j=1}^n\sum_{i=1}^n d_jd_ip\otimes
e_j\wedge e_i\wedge v=0\).
\end{proof}

\begin{corollary}\label{reducible}
Let \(P\) be a simple \(D_n\)-module. Then
\begin{enumerate}
\item \(\pi_k\neq0\) for \(0\leq k\leq n-1\).
\item\(T(P,V(\delta_r,r))\) is reducible for \(1\leq r\leq n-1\).
\end{enumerate}
\end{corollary}

\begin{proof}
(1) Suppose that $d_i(P)=0$ for some $1\leq i\leq n$. Then for any $p\in P$, we have
$$0=d_i(x_ip)=x_i(d_ip)+p=p,$$
which would imply \(P=0\), a contradiction. Thus \(d_i(P)\neq0\)
for every \(i\). From
$$\pi_k(p\otimes e_{i_1}\wedge\dots \wedge e_{i_k})= \sum_{i\neq i_1,\dots,i_k}d_ip\otimes e_i \wedge  e_{i_1}\wedge\dots \wedge e_{i_k},$$
we can see that \(\pi_k\neq0\) for \(0\leq k\leq n-1\).

(2)
By Lemma~\ref{homomorphismpi},
\(\pi_r\pi_{r-1}=0\) for \(1\leq r\leq n-1\). Hence
$
\operatorname{Im}\pi_{r-1}\subseteq\ker\pi_r.
$
By (1), both maps are nonzero.
Therefore \(\operatorname{Im}\pi_{r-1}\) is a nonzero proper
submodule of \(T(P,V(\delta_r,r))\).
\end{proof}

Considering the de Rham complex:
\begin{equation}\aligned \label{complex}
0 \rightarrow T(P,V(\delta_0,0)) &\xrightarrow{\pi_0} T(P,V(\delta_1,1)) \xrightarrow{\pi_1} \cdots\\
 &\xrightarrow{\pi_{n-2}} T(P,V(\delta_{n-1},n-1)) \xrightarrow{\pi_{n-1}} T(P,V(\delta_n,n)) \rightarrow \{0\}.
\endaligned\end{equation}
Let $L_n(P,r)=\pi_{r-1}(T(P,V(\delta_{r-1},r-1)))$ for $r=1,\dots, n$ and set $L_n(P,0)=0$. Note that $L_n(P,r)$ is spanned by
\begin{equation}\label{L_n}
\sum^n_{l=1}d_lp\otimes e_l\wedge e_{i_2}\wedge e_{i_3}\wedge \cdots \wedge e_{i_r}=\sum^n_{l=1}d_lp\otimes E_{l,j}v,
\end{equation}
where $p\in P$ and $v=e_j\wedge e_{i_2}\wedge e_{i_3}\wedge \cdots \wedge e_{i_r}\neq0$.\\
Let
$$\tilde{L}_n(P,r)=\{u\in T(P,V(\delta_r,r)) \mid \mg_nu\subseteq L_n(P,r)\},$$
which is the maximal submodule of $T(P,V(\delta_r,r))$
such that $\tilde{L}_n(P,r)/L_n(P,r)$ is a trivial $\mg_n$-module.
Clearly \(L_n(P,r)\subseteq\tilde L_n(P,r)\). Thus
$$\tilde{L}_n(P,r)\subseteq \{u\in T(P,V(\delta_r,r)) \mid \mg_nu\subseteq \tilde{L}_n(P,r)\}.$$
Since $\mg_n$ is a simple Lie algebra, $\mg_n=[\mg_n,\mg_n]$, we have
$$\tilde{L}_n(P,r)\supseteq \{u\in T(P,V(\delta_r,r)) \mid \mg_nu\subseteq \tilde{L}_n(P,r)\}.$$
Then
$$\tilde{L}_n(P,r)= \{u\in T(P,V(\delta_r,r)) \mid \mg_nu\subseteq \tilde{L}_n(P,r)\}.$$
.

\begin{lemma}\label{irreducible}
Let \(P\) be a simple \(D_n\)-module. Then
\[
T(P,V(\delta_r,r))/\tilde L_n(P,r)
\]
is either simple or zero for \(0\leq r\leq n\).
\end{lemma}
\begin{proof}
If $T(P,V(\delta_r,r))/\tilde{L}_n(P,r)\neq0$, let $\tilde{T}$ be a $\mg_n$-submodule of $T(P,V(\delta_r,r))$ with $\tilde{L}_n(P,r)\subsetneq \tilde{T}$. We only need to prove that
$\tilde{T}=T(P,V(\delta_r,r))$.

Take $v=\sum^q_{k=1}p_k\otimes v_k\in \tilde{T} \setminus \tilde{L}_n(P,r)$, then by the definition of $\tilde{L}_n(P,r)$, there exists some $t^{\bb}x^{\bs}d_j\in \mg_n$ such that $t^{\bb}x^{\bs}d_j v\notin \tilde{L}_n(P,r)$. For any $\ba\in \Z^n$, $\bm\in \Z^n_{\geq0}$, we have
\begin{equation}\label{mod L}
\aligned \
t^{\ba+\bb}x^{\bm+\bs}d_j v=&\sum^n_{i=1}\sum^q_{k=1}t^{\ba+\bb}x^{\bm+\bs}d_ip_k\otimes (\delta_{j,i}v_k-E_{i,j}v_k)\\
&+\sum^n_{i=1}\sum^q_{k=1}d_it^{\ba+\bb}x^{\bm+\bs}p_k\otimes E_{i,j}v_k,
\endaligned
\end{equation}
note that the second term belongs to $\tilde{L}_n(P,r)$, since (\ref{L_n}). Therefore
\begin{equation}\label{belong to V}
\sum^n_{i=1}\sum^q_{k=1}t^{\ba+\bb}x^{\bm+\bs}d_ip_k\otimes (\delta_{j,i}v_k-E_{i,j}v_k)\in \tilde{T}.
\end{equation}
When $\ba=\bm=\mathbf{0}$, (\ref{mod L}) implies
$$0\neq \sum^n_{i=1}\sum^q_{k=1}t^{\bb}x^{\bs}d_ip_k\otimes (\delta_{j,i}v_k-E_{i,j}v_k)\in \tilde{T}.$$
Applying $d_l(1\leq l\leq n)$ to (\ref{belong to V}), we have
$$\sum^n_{i=1}\sum^q_{k=1}d_lt^{\ba+\bb}x^{\bm+\bs}d_ip_k\otimes (\delta_{j,i}v_k-E_{i,j}v_k)\in \tilde{T}, \forall\ \ba\in \Z^n, \bm\in \Z^n_{\geq0}.$$
Then
$$\sum^n_{i=1}\sum^q_{k=1}ut^{\bb}x^{\bs}d_ip_k\otimes (\delta_{j,i}v_k-E_{i,j}v_k)\in \tilde{T}, \forall\  u\in D_n.$$
By the Jacobson density theorem,
\(P\otimes w\subseteq\tilde T\) for some
\(0\neq w\in V(\delta_r,r)\). Set
\[
V'=\{v'\in V(\delta_r,r)\mid P\otimes v'\subseteq\tilde T\}.
\]
Then \(V'\neq0\), and the argument used in Claim~2 of
Theorem~\ref{simplicity} shows that \(V'\) is a
\(\gl_n\)-submodule. Thus \(V'=V(\delta_r,r)\), and consequently
\(\tilde T=T(P,V(\delta_r,r))\).
\end{proof}

\begin{corollary}\label{ker}
Let $P$ be a simple $D_n$-module, then $\tilde{L}_n(P,r)=\ker \pi_r$ for all $r=0,1,\dots, n-1$.
\end{corollary}
\begin{proof}
Let \(v\in\tilde L_n(P,r)\), and suppose that \(\pi_r(v)\neq0\).
Write
\[
\pi_r(v)=\sum_{k=1}^q p_k\otimes v_k
\in T(P,V(\delta_{r+1},r+1)),
\]
where the \(p_k\) and the \(v_k\) are linearly independent,
respectively. For \(g\in\mg_n\), the definition of
\(\tilde L_n(P,r)\) and Lemma~\ref{homomorphismpi} give $
g\pi_r(v)=\pi_r(gv)=0$.
Taking \(g=d_j\), we obtain \(d_jp_k=0\) for all \(j,k\).
The map \(A_n\to P\), \(f\mapsto fp_1\), is then a nonzero
\(D_n\)-module homomorphism. Since the natural \(D_n\)-module \(A_n\)
is simple, it follows that \(P\cong A_n\). The joint kernel of
\(d_1,\ldots,d_n\) on \(A_n\) is \(\C\); hence the \(p_k\) are
linearly dependent. Thus \(q=1\), and
\(\pi_r(v)=p_1\otimes v'\) for some
\(v'\in V(\delta_{r+1},r+1)\). Consequently,
$$\aligned \
\pi_r(v)=&\frac{1}{r+1}(p_1\otimes Iv')\\
=&\frac{1}{r+1}(\sum^n_{i=1}(x_id_ip_1)\otimes v'+p_1\otimes Iv')\\
=&\frac{1}{r+1}\sum^n_{i=1}x_id_i(p_1\otimes v')
=\frac{1}{r+1}\sum^n_{i=1}x_id_i\pi_r(v)=
0
\endaligned$$
i.e., $\tilde{L}_n(P,r)\subseteq \ker \pi_r$.
By Corollary \ref{reducible}, $\text{Im}\pi_r\neq0$. Then from Lemma \ref{irreducible}, we have $\tilde{L}_n(P,r)=\ker \pi_r$ for all $r=0,1,\dots, n-1$.
\end{proof}

\begin{theorem}\label{remain}
Let \(P\) be a simple \(D_n\)-module. Then:
\begin{enumerate}
\item \(T(P,V(\delta_r,r))/\tilde{L}_n(P,r)\) is simple for
\(0\leq r\leq n-1\).

\item \(L_n(P,r)\) is simple for \(1\leq r\leq n\); moreover,
$$T(P,V(\delta_{r-1},r-1))/\tilde{L}_n(P,r-1)\cong L_n(P,r).$$

\item \(T(P,V(\delta_0,0))=P\) is simple if and only
 if $P\ncong A_n$. Moreover $A_n/\C t^{\bo}x^{\bo}$ is simple.

\item \(T(P,V(\delta_n,n))\) is simple if and only if
\(\sum_{i=1}^n d_iP=P\). If it is reducible, then
\(T(P,V(\delta_n,n))/L_n(P,n)\) is trivial.
\end{enumerate}
\end{theorem}

\begin{proof}
By Corollary \ref{ker}, $\ker\pi_r=\tilde{L}_n(P,r)$ for $r=0,1,\dots, n-1$. Since $\text{Im}\pi_r\neq0$, we have $T(P,V(\delta_r,r))/\tilde{L}_n(P,r)\neq0$. From Lemma \ref{irreducible}, (1) follows.

By the fundamental homomorphism theorem and (1), (2) follows.

If $P=A_n$, then $\C t^{\bo}x^{\bo}$ is a proper submodule
of $A_n$.
If $P\ncong A_n$, then there there is some $i$ such that \(d_i\) acts injectively on \(P\). For \(0\neq p\in P\),
we have \(d_i p\neq0\), and the simplicity of \(P\) as a
\(D_n\)-module gives \(P=D_n(d_i p)\). Then every \(p'\in P\) has the form
\[
p'=\sum_{j=1}^m\mu_j\nu_j(d_i p)
   =\left(\sum_{j=1}^m\mu_j(\nu_jd_i)\right)p,
\]
where \(\mu_j\in\C[d_1,\ldots,d_n]\) and \(\nu_j\in A_n\).
Since \(\nu_jd_i\in\mg_n\), the operator in parentheses belongs to
\(U(\mg_n)\). Thus \(U(\mg_n)p=P\), and hence $P$ is a simple $
\mg_n$-module. Note that $\tilde{L}_n(A_n,0)=\C t^{\bo}x^{\bo}$.
By (1) and (2), $A_n/\C t^{\bo}x^{\bo}$ is simple, proving part~\textup{(3)}.

 By (2), $L_n(P,n)$ is a simple submodule of $T(P,V(\delta_n,n))$. Hence $T(P,V(\delta_n,n))$ is simple if and only if $T(P,V(\delta_n,n))=L_n(P,n)$. We know that $L_n(P,n)=\sum^n_{i=1}d_iP\otimes V(\delta_n,n)$. Note that
 $E_{ii}v=v$ for any $v\in V(\delta_n,n)$. Then we can see that
$$\aligned
t^{\bb}x^{\bs}d_i(p\otimes v+L_n(P,n))
&=(t^{\bb}x^{\bs}d_ip)\otimes v
+(d_i(t^{\bb}x^{\bs}))p\otimes E_{ii}v+L_n(P,n))\\
&=d_i(t^{\bb}x^{\bs}p)\otimes v+L_n(P,n)=L_n(P,n).\endaligned$$
Therefore (4) follows.
\end{proof}

%
%
%

\section{Simple cuspidal $\mg_n$-modules}

In this section, we first we give a
tensor product decomposition of   $A_n\# U(\mg_n)$. Then we can obtain all
simple cuspidal $(A_n,\mg_n)$-modules. Finally using the technique of  $A$-cover,
we can show that any simple cuspidal $\mg_n$-module is isomorphic to
some simple quotient of  $T(A_n(\lambda), V)$.

 \subsection{Tensor product decomposition of   $A_n\# U(\mg_n)$}
Since $(A_n,\mg_n)$-modules are exactly modules over
 $A_n\# U(\mg_n)$,
we need to study the structure of
 $A_n\# U(\mg_n)$.

 We use $\cdot$ to denote the multiplication between $A_n$ and $U(\mg_n)$ in $A_n\# U(\mg_n)$. We can check that
 $A_n\cdot \mg_n$ is a Lie subalgebra of $(A_n\# U(\mg_n))_L$ and

\begin{equation}\label{bracket}
 \aligned \
 [f\cdot d, f'\cdot d']&=(f\cdot d)( f'\cdot d')-(f'\cdot d')( f\cdot d)\\
 &=fd(f')\cdot d'-f'd'(f)\cdot d+ff'\cdot [d,d'],
 \endaligned
\end{equation}
for any \(f,f'\in A_n\) and \(d,d'\in\mg_n\).

Let
\[\mm_{\b1,\bo}=\langle t_1-1,\dots,t_n-1,x_1,\dots,x_n\rangle\]
be the maximal ideal of \(A_n\) generated by
\[t_1-1,\dots,t_n-1,x_1,\dots,x_n.\]
Set $\Delta=\mathrm{span}\{d_1,\dots,d_n\}$.
By (\ref{bracket}), the subspace $$\mm_{\b1,\bo}\cdot \mg_n+A_n\cdot \mm_{\b1,\bo}\Delta$$
is a Lie subalgebra of   $A_n\cdot \mg_n$, and $\mm_{\b1,\bo}\cdot \mg_n$
is an ideal of
$\mm_{\b1,\bo}\cdot \mg_n+A_n\cdot \mm_{\b1,\bo}\Delta$. Moreover
\begin{equation}
\label{factor} (\mm_{\b1,\bo}\cdot \mg_n+A_n\cdot \mm_{\b1,\bo}\Delta)/\mm_{\b1,\bo}\cdot \mg_n\cong  \mm_{\b1,\bo}\Delta.
\end{equation}

In $A_n\# U(\mg_n)$,
we
define the following element:

\begin{equation}\label{eq:reconstruction}\aligned
 X_i(\ba,\bm)&=\sum_{\br=\bo}^{\bm} (-1)^{|\br|}\binom{\bm}{\br}t^{-\ba}x^{\br}\cdot t^{\ba}x^{\bm-\br}d_{i}
 -\delta_{\bm,\bo}d_i,
 \endaligned\end{equation}
for every $\ba\in \Z^n, \bm\in\Z_{\geq 0}^n$, $i=1,\dots,n$, where
 $\binom{\bm}{\br}=\binom{m_1}{r_1}\dots\binom{m_n}{r_n}$.
For example $ X_i(\ba,\bo)=t^{-\ba}\cdot t^{\ba}d_{i}-d_i$.
We can check that $X_i(\ba,\bm)\in \mm_{\b1,\bo}\cdot \mg_n+A_n\cdot \mm_{\b1,\bo}\Delta$.

 It turns out that these elements commute with
$t_i, x_i,d_i$.

 \begin{lemma}\label{D-L} For any $\ba\in \Z^n, \bm\in\Z_{\geq 0}^n$, $i,j=1,\dots,n$  we have
 $$[d_j,X_i(\ba,\bm)]=[t_j,X_i(\ba,\bm)]=[x_j,X_i(\ba,\bm)]=0.$$
\end{lemma}

\begin{proof}
We can compute that $$\aligned\
 [d_j, X_i(\ba,\bm)]
&= \sum_{\br=\bo}^{\bm} (-1)^{|\br|}\binom{\bm}{\br}
(-a_jt^{-\ba}x^{\br}+r_jt^{-\ba}x^{\br-e_j})\cdot t^{\ba}x^{\bm-\br}d_{i}\\
& \ \ +\sum_{\br=\bo}^{\bm} (-1)^{|\br|}\binom{\bm}{\br}t^{-\ba}x^{\br}\cdot (a_jt^{\ba}x^{\bm-\br}+(m_j-r_j)t^{\ba}x^{\bm-\br-e_j})d_{i}\\
&= \sum_{\br=\bo}^{\bm} (-1)^{|\br|}\binom{\bm}{\br}
r_jt^{-\ba}x^{\br-e_j}\cdot t^{\ba}x^{\bm-\br}d_{i}\\
& \ \ +\sum_{\br\geq e_j}^{\bm} -(-1)^{|\br|}\binom{\bm}{\br-e_j}(m_j-r_j+1)t^{-\ba}x^{\br-e_j}\cdot t^{\ba}x^{\bm-\br})d_{i}.
\endaligned$$
From $\binom{m_j}{r_j}
r_j=\binom{m_j}{r_j-1}(m_j-r_j+1)$, it follows that
 $[d_j, X_i(\ba,\bm)]=0$.

Then it can be obtained that  $$\aligned\
 [X_i(\ba,\bm),t_j]
&= \delta_{i,j}\sum_{\br=\bo}^{\bm} (-1)^{|\br|}\binom{\bm}{\br}t^{-\ba}x^{\br} t^{\ba+e_j}x^{\bm-\br}
 -\delta_{i,j}\delta_{\bm,\bo}t_j\\
&= \delta_{i,j}\sum_{\br=\bo}^{\bm} (-1)^{|\br|}\binom{\bm}{\br}t_jx^{\bm}
 -\delta_{i,j}\delta_{\bm,\bo}t_j\\
&=0,
 \endaligned$$
and  $$\aligned\
 [X_i(\ba,\bm),x_j]
&= \delta_{ij}\sum_{\br=\bo}^{\bm} (-1)^{|\br|}\binom{\bm}{\br}t^{-\ba}x^{\br} t^{\ba}x^{\bm-\br}
 -\delta_{i,j}\delta_{\bm,\bo}\\
&= \delta_{ij}\sum_{\br=\bo}^{\bm} (-1)^{|\br|}\binom{\bm}{\br} x^{\bm}
 -\delta_{i,j}\delta_{\bm,\bo}\\
&=0.
 \endaligned$$
\end{proof}

\begin{lemma}\label{dx}For every $\ba\in \Z^n, \bm\in\Z_{\geq 0}^n$, we have

\begin{equation}\label{X}\aligned
 t^{\ba}x^{\bm}d_{i}=\sum_{\bs=\bo}^{\bm} \binom{\bm}{\bs}t^{\ba}x^{\bs}X_i(\ba,\bm-\bs)
 +t^{\ba}x^{\bm}\cdot d_{i},
 \endaligned\end{equation}
where $i=1,\dots,n$.
\end{lemma}

\begin{proof}
From $$\binom{\bm}{\bs}
\binom{\bm-\bs}{\br}
=\binom{\bm}{\bs+\br}
\binom{\bs+\br}{\br}$$
and $$\sum_{\br=\bo}^{\bu} (-1)^{|\br|}\binom{\bu}{\br}=0, \bu>\bo,$$
it follows that
$$\aligned & \sum_{\bs=\bo}^{\bm} \binom{\bm}{\bs}t^{\ba}x^{\bs}X_i(\ba,\bm-\bs)\\
&= \sum_{\bs=\bo}^{\bm} \binom{\bm}{\bs}
\sum_{\br=\bo}^{\bm-\bs} (-1)^{|\br|}\binom{\bm-\bs}{\br}x^{\bs+\br}\cdot t^{\ba}x^{\bm-\bs-\br}d_{i}- t^{\ba}x^{\bm}\cdot d_{i}
\\
&= \sum_{\bs=\bo}^{\bm} \sum_{\br=\bo}^{\bm-\bs}\binom{\bm}{\bs+\br}
 (-1)^{|\br|}\binom{\bs+\br}{\br}x^{\bs+\br}\cdot t^{\ba}x^{\bm-\bs-\br}d_{i}
- t^{\ba}x^{\bm}\cdot d_{i}\\
&= t^{\ba}x^{\bm}d_{i}+\sum_{\bu>\bo}^{\bm} \binom{\bm}{\bu}
\sum_{\br=\bo}^{\bu} (-1)^{|\br|}\binom{\bu}{\br}x^{\bu}\cdot t^{\ba}x^{\bm-\bu}d_{i}- t^{\ba}x^{\bm}\cdot d_{i}
\\
&= t^{\ba}x^{\bm}d_{i}- t^{\ba}x^{\bm}\cdot d_{i}.
\endaligned$$
\end{proof}

Let
\[
T_n=\operatorname{span}\{X_i(\ba,\bm)\mid
\ba\in\Z^n,\ \bm\in\Z_{\geq0}^n,\ 1\leq i\leq n\}
\subseteq A_n\cdot\mg_n.
\]

\begin{lemma}\label{t-algebra}
\begin{enumerate}
\item $T_n=\{v\in A_n\cdot \mg_n\mid [v,A_n]=[v,\Delta]=0\}$. Thus
$T_n$ is a Lie subalgebra of $A_n\cdot \mg_n$.
\item The linear map $\gamma: T_n\rightarrow \mm_{\b1,\bo}\Delta,
X_i(\ba,\bm)\mapsto t^{\ba}x^{\bm}d_{i} -\delta_{\bm,\bo}d_i$ is a Lie algebra isomorphism.
\end{enumerate}
\end{lemma}

\begin{proof}
(1) Set $T'_n:=\{v\in A_n\cdot \mg_n\mid [v,A_n]=[v,\Delta]=0\}$. By
Lemma \ref{D-L}, we have $T_n\subset T'_n$. By Lemma \ref{dx}, $A_n\cdot \mg_n=A_n T_n\oplus A_n\cdot \Delta$.
So  for any $v\in T'_n$, we can write $v=\sum_{i,\ba,\bm} f_{i,\ba,\bm}X_i(\ba,\bm)+v'$, where $f_{i,\ba,\bm}\in A_n$,
$v'\in A_n\cdot \Delta$. For every \(a\in A_n\), one has
\([v,a]=[v',a]=0\). Since an element of \(A_n\cdot\Delta\) that
annihilates \(A_n\) is zero, it follows that \(v'=0\).
Then $[d_i,v]=[d_i,\sum_{i,\ba,\bm} f_{i,\ba,\bm}X_i(\ba,\bm)]=0$ implies that $f_{i,\ba,\bm}\in \C$, i.e., $v\in T_n$.
Thus $T_n=T'_n=\{v\in A_n\cdot \mg_n\mid [v,A_n]=[v,\Delta]=0\}$.

(2) As vector spaces, $T_n\cong\mm_{\b1,\bo}\Delta$. Since $X_i(\ba,\bm)\in \mm_{\b1,\bo}\cdot \mg_n+A_n\cdot \mm_{\b1,\bo}\Delta$ for any $\ba\in \Z^n, \bm\in\Z_{\geq 0}^n$, by (1), $T_n$ is a Lie subalgebra of $\mm_{\b1,\bo}\cdot\mg_n+A_n\cdot \mm_{\b1,\bo}\Delta$. We therefore have the Lie algebra homomorphism
$$T_n\subset\mm_{\b1,\bo}\cdot\mg_n+A_n\cdot \mm_{\b1,\bo}\Delta\rightarrow (\mm_{\b1,\bo}\cdot\mg_n+A_n\cdot \mm_{\b1,\bo}\Delta)/\mm_{\b1,\bo}\cdot\mg_n\cong\mm_{\b1,\bo}\Delta$$
that maps $X_i(\ba,\bm)$ to $t^{\ba}x^{\bm}d_{i} -\delta_{\bm,\bo}d_i$.
\end{proof}

Using the Lie algebra $T_n$, we  can give a tensor product decomposition
of $A_n\#U(\mg_n)$.

\begin{theorem}\label{mainth}
We have that $A_n\#U(\mg_n)=D_n U(T_n)\cong D_n\otimes U(T_n)$.
More precisely there is an
algebra isomorphism $\psi: D_n \otimes U(T_n)\rightarrow A_n\#U(\mg_n)$
such that $\psi |_{D_n}=\textup{Id}_{D_n}$ and $\psi |_{T_n}=\textup{Id}_{T_n}$,
whose inverse \(\phi:A_n\#U(\mg_n)\rightarrow D_n\otimes U(T_n)\)
is given by
$$\aligned\phi(t^{\ba}x^{\bm}d_{i})&=
\sum_{\br= \bo}^{\bm}\binom{\bm}{\br} t^{\ba}x^{\bm-\br} \otimes X_i(\ba,\br)+ (t^{\ba}x^{\bm}\cdot d_{i})\otimes 1, \\
\phi(d_i)&= d_i\otimes 1,\ \
\phi(t^{\ba}x^{\bm})=t^{\ba}x^{\bm}\otimes 1,\endaligned$$  where $\ba\in \Z^n, \bm\in\Z_{\geq 0}^n$, $i=1,\dots,n$.
\end{theorem}

\begin{proof}
Clearly \(D_nU(T_n)\subseteq A_n\#U(\mg_n)\), and the reverse
inclusion follows from Lemma~\ref{dx}. Hence
\[
A_n\#U(\mg_n)=D_nU(T_n).
\]
By Lemma~\ref{D-L}, \([D_n,T_n]=0\), so multiplication induces a
surjective algebra homomorphism
\[
\psi\colon D_n\otimes U(T_n)\longrightarrow A_n\#U(\mg_n).
\]

Define \(\phi\) on the generators \(A_n\cup\mg_n\) by the formulas in
the statement. Lemmas~\ref{D-L} and~\ref{t-algebra}, together with
the reconstruction formula \eqref{X}, show directly that these
assignments preserve the defining smash-product relations. Thus
\(\phi\) extends to an algebra homomorphism. Formula~\eqref{X} gives
\[
\psi\phi=\operatorname{Id}_{A_n\#U(\mg_n)}.
\]
Conversely, the definition of \(X_i(\ba,\bm)\) gives
\[
\phi\psi=\operatorname{Id}_{D_n\otimes U(T_n)}
\]
on the generating subalgebras \(D_n\otimes1\) and \(1\otimes T_n\).
Therefore \(\phi\) and \(\psi\) are mutually inverse.
\end{proof}

By Theorem \ref{mainth} and (2) in Lemma \ref{t-algebra}, we have the following isomorphism.

\begin{theorem}\label{mainth1}
There is an isomorphism
$$\phi_1:A_n\#U(\mg_n)\rightarrow D_n\otimes U(\mm_{\b1,\bo}\Delta)
$$
given by
$$\aligned\phi_1(t^{\ba}x^{\bm}d_{i})&=
\sum_{\br= \bo}^{\bm}\binom{\bm}{\br} t^{\ba}x^{\bm-\br} \otimes (t^{\ba}x^{\br}d_{i} -\delta_{\br,\bo}d_i)+ (t^{\ba}x^{\bm}\cdot d_{i})\otimes 1, \\
\phi_1(d_i)&= d_i\otimes 1,\ \
\phi_1(t^{\ba}x^{\bm})=t^{\ba}x^{\bm}\otimes 1,\endaligned$$  where $\ba\in \Z^n, \bm\in\Z_{\geq 0}^n$, $i=1,\dots,n$.
\end{theorem}

\subsection{Simple cuspidal $(A_n,\mg_n)$-modules}
In this subsection, we will describe all simple cuspidal $(A_n,\mg_n)$-modules
using Theorem \ref{mainth1}. First we study finite-dimensional simple $\mm_{\b1,\bo}\Delta$-modules.

\begin{lemma}\label{Hu}
Let \(V\) be finite-dimensional, and let
\(L\subseteq\gl(V)\) be a nonzero Lie algebra acting irreducibly on
\(V\). Then
\[
L=[L,L]\oplus Z(L),
\]
where \([L,L]\) is semisimple and \(\dim Z(L)\leq1\).
\end{lemma}

\begin{proof}
This is \cite[Proposition~19.1]{Hu}.
\end{proof}

\begin{lemma}\label{lem:finite-jet}
If \(V\) is a finite-dimensional \(\mathfrak m_{\b1,\bo}\Delta\)-module, then
there exists \(N\geq 2\) such that
\[
\mm_{\b1,\bo}^N\Delta V=0.
\]
\end{lemma}

\begin{proof}
For convenience,
let $\mm=\mm_{\b1,\bo}$, $L=\mathfrak m\Delta$,  $\rho: L\rightarrow\mathfrak{gl}(V)$ be the associated Lie homomorphism, and
$K=\ker\rho$.

\

\noindent \textbf{Claim 1:}
There exists \(N\in \Z_{\geq2}\) such that
\begin{equation}\label{eq1} x_i^{r}d_i V=0,\quad \forall\ 1\leq i\leq n, \forall\ r\in \Z_{\geq N}. \end{equation}

 For a fixed $1\leq i\leq n$, let $W_i=\text{span}\{L_r=x_i^rd_i\mid r\geq1\}$ which is a subalgebra of $L$. Then
 $$[x_id_i,x_i^{r}d_i]=(r-1)x_i^{r}d_i.$$
 Applying $\rho$, we obtain
 $$[\rho(x_id_i),\rho(x_i^{r}d_i)]=(r-1)\rho(x_i^{r}d_i).$$
 Thus if $\rho(x_i^{r}d_i)\neq 0$, then
 then \(r-1\) is an eigenvalue of $ \operatorname{ad}\rho(x_id_i) \colon \mathfrak{gl}(V)\longrightarrow\mathfrak{gl}(V)$. Since \(\mathfrak{gl}(V)\) is finite-dimensional, \(\operatorname{ad}\rho(x_id_i)\) has only finitely many eigenvalues. Consequently, there exits $N_i\in \Z_{\geq 2}$ such
 that  $x_i^{r}d_i V=0$, for all $r\in \Z_{\geq N_i}$. Let $N=\text{max}\{N_1,\dots,N_n\}$. Then Claim 1 holds.

\

\noindent \textbf{Claim 2:}
For any $b\in \Z$, we have
\begin{equation}\label{eq1} t_i^{b}x_i^{r}d_i V=0,\quad \forall\ 1\leq i\leq n,   \forall\ r\in \Z_{\geq N}. \end{equation}
where the integer \(N\in \Z_{\geq2 }\) is defined in Claim 1.

 For any $b\in \Z$ and $r\geq N$, by Claim 1 and that $\ker\rho$ is an ideal of $L$, we have
\begin{equation}\label{D_7}
\aligned \
[x_i^{r}d_i,t_i^bx_id_i]=bt_i^bx_i^{r+1}d_i+(1-r)t_i^bx_i^rd_i\in \ker\rho
\endaligned
\end{equation}
and
\begin{equation}\label{D_8}
\aligned \
[x_i^{r+1}d_i,(t_i^b-1)d_i]=&bt_i^bx_i^{r+1}d_i-(r+1)(t_i^b-1)x_i^rd_i\\
=&bt_i^bx_i^{r+1}d_i-(r+1)t_i^bx_i^rd_i+(r+1)x_i^rd_i\in \ker\rho.
\endaligned
\end{equation}
 Subtracting (\ref{D_7}) from (\ref{D_8}), by $x_i^rd_i\in \ker\rho$, we can see that Claim 2 holds.

\

\noindent \textbf{Claim 3:} For any $b\in \Z$, we have
\begin{equation}\label{eq1} t_i^{b}y_i^qx_i^{r}d_i V=0,\quad \  \forall\ q, r\in \Z_{\geq 0}\
\text{such that}\ q+r\geq N, \end{equation}
where $y_i=t_i-1$, where the integer \(N\in \Z_{\geq2 }\) is defined in Claim 1.

For convenience, denote $F_{b,q,r}=t_i^by_i^qx_i^rd_i$. Then we have
\begin{equation}\label{F_1}
\aligned \
&[F_{b,q,r},y_ix_id_i]-[F_{b,q,r+1},y_id_i]\\
=&t_i^{b+1}y_i^qx_i^{r+1}d_i+t_i^{b}y_i^{q+1}x_i^{r}d_i\\
&-bt_i^{b}y_i^{q+1}x_i^{r+1}d_i-qt_i^{b+1}y_i^{q}x_i^{r+1}d_i
-rt_i^{b}y_i^{q+1}x_i^{r}d_i\\
&-t_i^{b+1}y_i^{q}x_i^{r+1}d_i+bt_i^{b}y_i^{q+1}x_i^{r+1}d_i\\
&+qt_i^{b+1}y_i^{q}x_i^{r+1}d_i+(r+1)t_i^{b}y_i^{q+1}x_i^{r}d_i\\
=&2F_{b,q+1,r}
\endaligned
\end{equation}
and
\begin{equation}\label{F_2}
\aligned \
[F_{b,q,r},y_id_i]=&t_i^{b+1}y_i^{q}x_i^{r}d_i-y_i(bt_i^{b}y_i^{q}x_i^{r}d_i+qt_i^{b+1}y_i^{q-1}x_i^{r}d_i+rt_i^{b}y_i^{q}x_i^{r-1}d_i)\\
=&(1-q)F_{b+1,q,r}-bF_{b,q+1,r}-rF_{b,q+1,r-1}.
\endaligned
\end{equation}
Since $F_{b,0,r}\in \ker\rho$ for any $r\geq N$, by (\ref{F_1}) and induction on $q$,  we have
$$\aligned \
F_{b,q,r}\in \ker\rho, \ \forall\ q\in \Z_{\geq 0}, \ r\in \Z_{\geq N}, b\in \Z.
\endaligned$$
Then  from (\ref{F_2}), by descending induction on $r$, one can obtain
\begin{equation}\label{F_3}
t_i^by_i^qx_i^rd_i\in \ker\rho, \ \forall\  q+r\geq N.
\end{equation}
Hence Claim 3 holds.

\

\noindent \textbf{Claim 4:}
There exists $N'\geq nN+1$ such that
\begin{equation} t^{\mathbf{b}}y^{\mathbf{q}}x^{\mathbf{r}}d_i V=0,\quad \forall\ |\mathbf{q}|+|\mathbf{r}|\geq N', \forall\ i=1,\dots,n, \end{equation}
where $y^{\mathbf{q}}=y_1^{q_1}\dots y_n^{q_n}$, $y_i=t_i-1$.

If $n=1$, then it is case (\ref{F_3}). Next, we assume $n\geq 2$. 

 For $\mathbf{q},\mathbf{r}\in \Z_{\geq 0}^n$ such that  $|\mathbf{q}|+|\mathbf{r}|\geq N'$, there must exist some $i$ such that $q_i+r_i\geq N+1$.  By Claim 3, $t_i^{b_i}y_i^{q_i}x_i^{r_i}d_i\in \ker \rho$. Consequently  for $l\neq i$, we have
\begin{equation}\label{F_4}
[t_i^{b_i}y_i^{q_i}x_i^{r_i}d_i,x_i
t^{\mathbf{b}-b_ie_i}y^{\mathbf{q}-q_ie_i}x^{\mathbf{r}-r_ie_i}d_l]
=t^{\mathbf{b}}y^{\mathbf{q}}x^{\mathbf{r}}d_l\in \ker\rho.
\end{equation}
For $l=i$, since $n\geq2$, we take $j\neq i$, then by (\ref{F_4}) we know that $t^{\mathbf{b}}y^{\mathbf{q}}x^{\mathbf{r}}d_j\in \ker\rho$.
Consequently
\begin{equation}\label{4.15}[t^{\mathbf{b}}y^{\mathbf{q}}x^{\mathbf{r}}d_j, x_jd_i]=t^{\mathbf{b}}y^{\mathbf{q}}x^{\mathbf{r}}d_i-x_jd_i(t^{\mathbf{b}}y^{\mathbf{q}}x^{\mathbf{r}})d_j\in \ker\rho.\end{equation}
One can see that the sum of the degrees of $y_i$ and $x_i$ in $d_i(t^{\bb}y^{\mathbf{q}}x^{\br})$ is at least $N$, similar to the argument in (\ref{F_4}), we have
$$x_jd_i(t^{\mathbf{b}}y^{\mathbf{q}}x^{\mathbf{r}})d_j\in \ker\rho.$$
 Then by (\ref{4.15}), $t^{\mathbf{b}}y^{\mathbf{q}}x^{\mathbf{r}}d_i\in \ker\rho$.  Note that $\mm^{N'}\Delta$ is spanned by $t^{\mathbf{b}}y^{\mathbf{q}}x^{\mathbf{r}}d_l$ with $|\mathbf{q}|+|\mathbf{r}|\geq N'$. Therefore $\mm^{N'}\Delta \subseteq \ker\rho$.

Now we have established this lemma.
\end{proof}

\begin{lemma}\label{lem:gln-quotient1}
The following statements hold.

\begin{enumerate}
\item
Let $
\pi_1\colon
\mathfrak m_{1,0}\Delta
\longrightarrow
\mathfrak m_{1,0}\Delta/
\mathfrak m_{1,0}^{2}\Delta
$
be the canonical projection. For every
\(\ba\in\mathbb Z^n\), we have
\[
\pi_1(t^\ba d_i-d_i)
 =
 \sum_{k=1}^{n}a_k(t_k-1)d_i
 +\mathfrak m_{1,0}^{2}\Delta,
\]
\[
\pi_1(t^\ba x_kd_i)
 =
 x_kd_i+\mathfrak m_{1,0}^{2}\Delta,
\]
and
\[
\pi_1(t^\ba x^\br d_i)=0,
\qquad\text{whenever } |\br|\geq 2.
\]

\item
The linear map
\[
\pi_2\colon
\mathfrak m_{1,0}\Delta/
\mathfrak m_{1,0}^{2}\Delta
\rightarrow
\mathfrak{gl}_n
\]
defined by
\[
x_kd_i+\mathfrak m_{1,0}^{2}\Delta
\mapsto E_{ki},
\qquad
(t_k-1)d_i+\mathfrak m_{1,0}^{2}\Delta
\mapsto E_{ki},
\]
is a surjective Lie algebra homomorphism.
\end{enumerate}
\end{lemma}

\begin{proof}
For convenience, we still denote
\[
y_k=t_k-1,
\qquad
\mathfrak m=\mathfrak m_{1,0},
\qquad
L=\mathfrak m\Delta,
\qquad
J=\mathfrak m^2\Delta.
\]
For \(f,g\in A_n\), the bracket in \(A_n\Delta\) is given by
\begin{equation}\label{eq:basic-vector-field-bracket}
[fd_i,gd_j]
 =
 f\,d_i(g)d_j-g\,d_j(f)d_i.
\end{equation}
Since $
d_i(\mathfrak m^q)\subseteq \mathfrak m^{q-1}$,
we have
\begin{equation}\label{eq:filtration-bracket}
[\mathfrak m^p\Delta,\mathfrak m^q\Delta]
\subseteq
\mathfrak m^{p+q-1}\Delta
\end{equation}
for all \(p,q\geq 1\). In particular, \(J\) is an ideal of \(L\).

\smallskip

(1)
For any \(a_k\in\mathbb Z\), we have
\[
(1+y_k)^{a_k}\equiv 1+a_ky_k\pmod{y_k^2}.
\]
This congruence also holds when \(a_k<0\), since the formal expansion
of \((1+y_k)^{a_k}\), modulo \(y_k^2\), still has constant term \(1\)
and linear term \(a_ky_k\). Therefore,
\[
t^a
 =
 \prod_{k=1}^{n}(1+y_k)^{a_k}
 \equiv
 1+\sum_{k=1}^{n}a_ky_k
 \pmod{\mathfrak m^2}.
\]
It follows that $
t^\ba-1-\sum_{k=1}^{n}a_k(t_k-1)\in\mathfrak m^2$
and hence $$
\pi_1(t^\ba d_i-d_i)
 =
 \sum_{k=1}^{n}a_k(t_k-1)d_i+J.$$

Moreover, since \(t^\ba-1\in\mathfrak m\), we have
\(
t^\ba x_k-x_k=(t^\ba-1)x_k\in\mathfrak m^2
\).
Thus $
\pi_1(t^\ba x_kd_i)=x_kd_i+J$.

Finally, if \(|\br|\geq2\), then \(x^\br\in\mathfrak m^2\).
Consequently, $
t^\ba x^\br d_i\in\mathfrak m^2\Delta=J$,
which proves
$
\pi_1(t^\ba x^\br d_i)=0$.

\smallskip

(2)
Set
$
Y_{ki}=(t_k-1)d_i+J=y_kd_i+J$,
and $
X_{ki}=x_kd_i+J$.
Since $$
\mathfrak m/\mathfrak m^2
 =
 \operatorname{span}_{\mathbb C}
 \{y_1,\ldots,y_n,x_1,\ldots,x_n\},$$
the elements \(Y_{ki}\) and \(X_{ki}\) form a spanning set of \(L/J\).

We calculate their brackets. First,
\[
d_i(y_l)=d_i(t_l-1)=\delta_{il}t_l
=\delta_{il}(1+y_l).
\]
Using \eqref{eq:basic-vector-field-bracket}, we obtain
\begin{align*}
[y_kd_i,y_ld_j]
&=
y_kd_i(y_l)d_j-y_ld_j(y_k)d_i\\
&=
\delta_{il}y_k(1+y_l)d_j
-
\delta_{jk}y_l(1+y_k)d_i.
\end{align*}
Modulo \(J\), all terms of degree at least two vanish. Hence
\begin{equation}\label{eq:YY-bracket}
[Y_{ki},Y_{lj}]
 =
 \delta_{il}Y_{kj}-\delta_{jk}Y_{li}.
\end{equation}

Since \(d_i(x_l)=\delta_{il}\), we similarly have
\begin{equation}\label{eq:XX-bracket}
[X_{ki},X_{lj}]
 =
 \delta_{il}X_{kj}-\delta_{jk}X_{li}.
\end{equation}

For the mixed bracket,
\begin{align*}
[y_kd_i,x_ld_j]
&=
y_kd_i(x_l)d_j-x_ld_j(y_k)d_i\\
&=
\delta_{il}y_kd_j-\delta_{jk}x_l(1+y_k)d_i.
\end{align*}
Since \(x_ly_kd_i\in J\), it follows that
\begin{equation}\label{eq:YX-bracket}
[Y_{ki},X_{lj}]
 =
 \delta_{il}Y_{kj}-\delta_{jk}X_{li}.
\end{equation}

Equations \eqref{eq:YY-bracket}--\eqref{eq:YX-bracket} show that
\(\pi_2\) preserves all brackets. Therefore \(\pi_2\) is a Lie
algebra homomorphism. It is clearly surjective.
\end{proof}

The following key lemma classifies all finite-dimensional
simple \(\mathfrak m_{1,0}\Delta\)-modules.

\begin{lemma}\label{lem:gln-quotient}
If \(V\) is a finite-dimensional simple
\(\mathfrak m_{1,0}\Delta\)-module, then
\[
\mathfrak m_{1,0}^{2}\Delta V=0
\]
and
\[
(t_k-1-x_k)d_lV=0
\qquad
\text{for all }1\leq k,l\leq n.
\]
Consequently, the action of \(\mathfrak m_{1,0}\Delta\) on \(V\)
factors through \(\mathfrak{gl}_n\), and \(V\) is a simple
\(\mathfrak{gl}_n\)-module.
\end{lemma}

\begin{proof}
Let
\(
\rho\colon L\rightarrow\mathfrak{gl}(V)
\)
be the representation associated with \(V\).

By Lemma~\ref{lem:finite-jet}, there exists an integer \(N\geq2\)
such that
\begin{equation}\label{eq:finite-jet-property}
\mathfrak m^N\Delta V=0.
\end{equation}
Equivalently,  \(V\) is a module over the  quotient $
L/\mathfrak m^N\Delta$.

We first prove that
$
JV=\mathfrak m^2\Delta V=0$.
Let \(J^{(1)}=J\) and define inductively $
J^{(q+1)}=[J,J^{(q)}]$.
It follows from \eqref{eq:filtration-bracket}, by induction on \(q\),
that
\[
J^{(q)}\subseteq\mathfrak m^{q+1}\Delta.
\]
Together with \eqref{eq:finite-jet-property}, this implies that
\(\rho(J)\) is a nilpotent, and hence solvable, Lie algebra.

Since \(J\) is an ideal of \(L\), its image \(\rho(J)\) is an ideal
of $
G:=\rho(L)$.
The \(G\)-module \(V\) is irreducible. Therefore, by Lemma \ref{Hu},
\[
G=[G,G]\oplus Z(G),
\]
where \([G,G]\) is semisimple, and $Z(G)$ is equal to
the solvable radical of $G$. So as a  solvable ideal of $G$, we have
$\rho(J)\subseteq Z(G)$.
In particular,
\begin{equation}\label{eq:LJ-zero}
\rho([L,J])=0.
\end{equation}

We first eliminate the pure \(x\)-terms. Set
$ E_x=\sum_{i=1}^n x_id_i\in L$.
If $f\in\C[x_1,\ldots,x_n]$ is homogeneous of degree \(r\geq2\),
then
\[
[E_x,fd_j]=(r-1)fd_j.
\]
Since \(fd_j\in J\), equation~\eqref{eq:LJ-zero} gives
\(\rho(fd_j)=0\). Hence \(\mathfrak m_x^2\Delta V=0\), where
 $\mm_x
 =
 \langle x_1,\ldots,x_n\rangle
 \subseteq
 \mathbb C[x_1,\ldots,x_n]$.

For the pure \(t\)-terms, let \(y_i=t_i-1\),
\(\partial_i=t_i^{-1}d_i\), and
$E_t=\sum_{i=1}^n y_i\partial_i\in L$.
If \(f\in\C[y_1,\ldots,y_n]\) is homogeneous of degree \(r\geq2\),
then
\[
[E_t,f\partial_j]=(r-1)f\partial_j.
\]
Thus every such \(f\partial_j\) acts trivially on $V$. Let $
\mathfrak m_t
 =
 \langle t_1-1,\ldots,t_n-1\rangle
 \subseteq
 \mathbb C[t_1^{\pm1},\ldots,t_n^{\pm1}]
$. Modulo
\(\mathfrak m_t^N\), each Laurent monomial in the \(t_i\) has a
finite Taylor expansion in the \(y_i\). Together with
\eqref{eq:LJ-zero}, this shows that every element of
\(\mathfrak m_t^2\Delta\) acts trivially. Therefore
\begin{equation}\label{eq:pure-terms-zero}
\mathfrak m_t^2\Delta V=0,
\qquad
\mathfrak m_x^2\Delta V=0.
\end{equation}

We next prove that all mixed terms act trivially. More precisely,
we show that
\begin{equation}\label{eq:mixed-zero}
y^\ba x^\bs d_lV=0
\end{equation}
whenever \(|\ba|\geq1\) and \(|\bs|\geq1\).

By \eqref{eq:finite-jet-property}, the assertion is true when
\(|\ba|+|\bs|\geq N\). We now argue by descending induction on the total
degree
$q=|\ba|+|\bs|$,
and, for a fixed $q$, by ascending induction on \(|\ba|\).

The case \(|\ba|=0\) follows from
\(\mathfrak m_x^2\Delta V=0\). Suppose that \(|\ba|\geq1\).
For any \(1\leq i\leq n\), formula
\eqref{eq:basic-vector-field-bracket} gives
\begin{align*}
[y^\ba d_i,x^{\bs+e_i}d_l]
={}&
(s_i+1)y^\ba x^\bs d_l\\
&-
a_l y^{\ba-e_l}x^{\bs+e_i}d_i
-
a_l y^\ba x^{\bs+e_i}d_i.
\end{align*}
The second element \(x^{\bs+e_i}d_l\) in the bracket
 belongs to \(\mathfrak m_0^2\Delta\), and therefore
acts trivially by \eqref{eq:pure-terms-zero}. Hence the left-hand
side acts trivially on \(V\).
The term
$y^{\ba-e_l}x^{\bs+e_i}d_i$
has the same total degree \(q\), but has \(y\)-degree \(|\ba|-1\);
it acts trivially by the induction hypothesis on \(|\ba|\).
The term $y^\ba x^{\bs+e_i}d_i$
has total degree \(q+1\), and acts trivially by descending induction
on \(q\). Therefore
$(s_i+1)y^ax^sd_lV=0$.
Since \(s_i+1\neq0\),  \eqref{eq:mixed-zero} follows.

Every coefficient in \(\mathfrak m^2\), modulo \(\mathfrak m^N\),
is a linear combination of monomials of the following three types:
\[
y^\ba,\quad |\ba|\geq2;
\qquad
x^\bs,\quad |\bs|\geq2;
\qquad
y^\ba x^\bs,\quad |\ba|,|\bs|\geq1.
\]
The first two types act trivially by
\eqref{eq:pure-terms-zero}, and the third type acts trivially by
\eqref{eq:mixed-zero}. Together with
\eqref{eq:finite-jet-property}, this proves
\begin{equation}\label{eq:J-zero}
JV=\mathfrak m^2\Delta V=0.
\end{equation}

Thus \(V\) is naturally a simple module over \(L/J\). Define
\[
I=
\operatorname{span}_{\mathbb C}
\left\{
(t_k-1-x_k)d_l+J
\ \middle|\
1\leq k,l\leq n
\right\}
\subseteq L/J.
\]
From $$\aligned \
&[(t_i-1)d_j+\mm^2\Delta,(t_k-1-x_k)d_l+\mm^2\Delta]\\
&=\delta_{jk}(t_i-1)(t_k-1)d_l
-\delta_{li}(t_k-1-x_k)t_id_j+\mm^2\Delta\\
&=-\delta_{li}(t_k-1-x_k)d_j+\mm^2\Delta
\endaligned $$
and
$$\aligned \
&[x_id_j+\mm^2\Delta,(t_k-1-x_k)d_l+\mm^2\Delta]\\
&=\delta_{jk}x_i(t_k-1)d_l
-\delta_{li}(t_k-1-x_k)d_j+\mm^2\Delta\\
&=-\delta_{li}(t_k-1-x_k)d_j+\mm^2\Delta,
\endaligned $$
it follows that $[L/J,I]=I$ and $[I,I]=0$.

Let
$\overline{\rho}\colon L/J\rightarrow\mathfrak{gl}(V)$
be the induced representation and
$\overline{G}=\overline{\rho}(L/J)$.
Since \(I\) is an abelian ideal, \(\overline{\rho}(I)\) is a solvable
ideal of the reductive Lie algebra \(\overline{G}\). Hence
\[
\overline{\rho}(I)\subseteq Z(\overline{G}).
\]
Then we obtain
$\overline{\rho}(I)
 =
 \overline{\rho}([L/J,I])
 =
 [\overline{G},\overline{\rho}(I)]
 =
 0$.
Consequently,
$(t_k-1-x_k)d_lV=0$ for all $1\leq k,l\leq n$.
Finally, the kernel of the homomorphism \(\pi_2\) is
$ I=
\operatorname{span}_{\mathbb C}
\{Y_{ki}-X_{ki}\mid 1\leq k,i\leq n\}$ .
Therefore, $\frac{L/J}{I}\cong\mathfrak{gl}_n$.
The action of \(L\) on \(V\) factors through this quotient, so \(V\)
is a \(\mathfrak{gl}_n\)-module.
\end{proof}

The following result is well-known.

\begin{lemma}\label{tensor proper}
Let $A$, $B$ be two unital associative algebras with B having a countable basis.
\begin{enumerate}
\item If $M$ is a simple module over $A\otimes B$ that contains a simple $B=\C\otimes B$ submodule $V$, then $M \cong W\otimes V$ for a simple $A$-module $W$.
\item If $W$ and $V$ are simple modules over $A$ and $B$ respectively, then $W\otimes V$ is a simple module over $A\otimes B$.
\end{enumerate}
\end{lemma}

For \(\lambda\in\C^n\), set \(A_n(\lambda)=t^\lambda A_n\), equipped
with the natural \(D_n\)-action by differential operators. Then
\(A_n(\lambda)\cong D_n/I_\lambda\), where \(I_\lambda\) is the left
ideal of \(D_n\) generated by \(d_i-\lambda_i\), \(i=1,\ldots,n\).

\begin{lemma}\label{A(lambda)}
\begin{enumerate}
\item $A_n(\lambda)$ is a simple $D_n$-module.
\item Any simple generalized weight $D_n$-module is isomorphic to some
$A_n(\lambda)$ for some $\lambda \in \C^n$.
\end{enumerate}
\end{lemma}

\begin{proof}
(1) Since
\[
d_it^{\lambda+\ba}x^{\bm}
=(\lambda_i+a_i)t^{\lambda+\ba}x^{\bm}
+m_it^{\lambda+\ba}x^{\bm-e_i},
\]
we have
\begin{equation}\label{d-act}
(d_i-\lambda_i-a_i)^{m_i}t^{\lambda+\ba}x^{\bm}
=m_i!\,t^{\lambda+\ba}x^{\bm-m_ie_i},
\end{equation}
and $(d_i-\lambda_i-a_i)^{m_i+1}t^{\lambda+\ba}x^{\bm}=0$.
Then $A_n(\lambda)=\oplus_{\ba\in \Z^n} A_n(\lambda)(\lambda+\ba)$,
where $A_n(\lambda)(\lambda+\ba)=t^{\lambda+\ba}\C[x_1,\dots,x_n]$
is the generalized weight space with weight $\lambda+\ba$. Suppose that $V$ is a
nonzero $D_n$-submodule of $A_n(\lambda)$. Since submodules of a generalized
weight module are also generalized
weight modules, $V$ must contain some $t^{\lambda+\ba}f(x_1,\dots,x_n)$.
By the action $t^{-\ba}$, $t^{\lambda}f(x_1,\dots,x_n)\in V$. Using (\ref{d-act})
repeatedly, we can obtain that the generator $t^\lambda$ of $A_n(\lambda)$
belongs to $V$. Hence $V=A_n(\lambda)$.

(2) Let \(M\) be a simple generalized weight \(D_n\)-module and
choose \(\lambda\in\operatorname{supp}M\). Since the commuting
operators \(d_i-\lambda_i\) act locally nilpotently on \(M(\lambda)\),
this space contains a nonzero common eigenvector \(v\in M_\lambda\).
Then \(D_nv=A_n\C[d_1,\dots,d_n]v=A_nv\). Hence the map
\[
u t^\lambda\mapsto uv, \quad u\in A_n,
\]
defines a nonzero \(D_n\)-module homomorphism
from $A_n(\lambda)$ to $M$. Part~\textup{(1)} and the simplicity of \(M\)
show that this map is an isomorphism.
\end{proof}

\begin{proposition}\label{(A_n,g_n)-module}
We have the following assertions.
\begin{enumerate}
\item For any \(\lambda\in\C^n\) and simple
\(\mm_{\b1,\bo}\Delta\)-module \(V\), the module
\(T(A_n(\lambda),V)\) is a simple \((A_n,\mg_n)\)-module.
\item For any simple cuspidal \((A_n,\mg_n)\)-module \(M\), there is a
finite-dimensional simple \(\gl_n\)-module \(V\) such that
\(M\cong T(A_n(\lambda),V)\).
\end{enumerate}
\end{proposition}

\begin{proof}
Part~\textup{(1)} follows from Lemma~\ref{tensor proper}(2) and
Theorem~\ref{mainth1}.

For part~\textup{(2)}, use \(\phi_1\) to regard \(M\) as a module over
\[
D_n\otimes U(\mm_{\b1,\bo}\Delta).
\]
Choose a weight \(\lambda\) and \(0\neq v\in M_\lambda\). Then
\[
D_nv=A_nv\cong A_n(\lambda)
\]
by Lemma~\ref{A(lambda)}. Applying Lemma~\ref{tensor proper}(1), we have
\[
M\cong A_n(\lambda)\otimes V
\]
for a simple \(U(\mm_{\b1,\bo}\Delta)\)-module \(V\). Since
the weight space $\bigl(A_n(\lambda)\otimes V\bigr)_{\lambda+\ba}
=t^{\lambda+\ba}\otimes V$ is finite-dimensional,
it follows  that \(V\) is finite-dimensional.
By Lemma~\ref{lem:gln-quotient},   \(V\) is a
finite-dimensional simple \(\gl_n\)-module.
\end{proof}

\subsection{Simple cuspidal $\mg_n$-module}

In this subsection, we will use the technique of  $A$-cover introduced in \cite{BF1} to classify simple cuspidal $\mg_n$-modules.
Let $\Theta$ be the co-multiplication of $U(\widetilde{\mg}_n)$, that is
$$\Theta: U(\widetilde{\mg}_n)\rightarrow U(\widetilde{\mg}_n)\otimes U(\widetilde{\mg}_n), X\mapsto X\otimes 1+1\otimes X,\  \ X\in \widetilde{\mg}_n.$$
Let $\pi: U(\widetilde{\mg}_n)\rightarrow U(\mg_n)$ be the homomorphism such that
$\pi |_{\mg_n}=\textup{Id}_{\mg_n}, \pi |_{A_n}=0$. Denote the composition $$(\textup{Id}_{\widetilde{\mg}_n}\otimes \pi)\cdot\Theta: U(\widetilde{\mg}_n)\rightarrow U(\widetilde{\mg}_n)\otimes U(\mg_n)$$ by $\Theta'$.
Then for any $(A_n,\mg_n)$-module $P$ and $\mg_n$-module $M$, $P\otimes M$ becomes an $(A_n,\mg_n)$-module under the pulling back of $\Theta'$. In particular,
$\mg_n\otimes M$ can be defined to be an $(A_n,\mg_n)$-module such that
$$t^{\ba}x^\bm  (t^{\bb}x^{\bs}d_i\otimes w)=t^{\ba+\bb}x^{\bm+\bs}d_i\otimes w,$$
where $w\in M$, $\ba,\bb\in \Z^n$,$\bm, \bs\in\Z_{\geq 0}^n$, $i=1,\dots,n$.

Define a linear map $\theta: \mg_n\otimes M\rightarrow M$ such that
$$t^{\ba}x^\bm d_i\otimes w\mapsto t^{\ba}x^\bm d_i w, \ w\in M.$$
One can check directly that $\theta$ is a $\mg_n$-module homomorphism.
Define
$$K(M)=\{ v\in \ker\theta \mid A_n  v\subset \ker \theta\},$$
which is an $(A_n,\mg_n)$-submodule of $\mg_n\otimes M$.
Let
$$\widehat{M}=(\mg_n\otimes M)/K(M),$$
 which is an $(A_n,\mg_n)$-module, called the $A$-cover of
$M$. It is natural that $\theta$ induces a $\mg_n$-module homomorphism
$\hat{\theta}: \widehat{M}\rightarrow M$. For any $X\in \mg_n$, $v\in M$,
denote the image of $X\otimes v$ in $\widehat{M}$ by $X\boxtimes v$.

%

\begin{proposition}\label{M-hat}
For any  cuspidal $\mg_n$-module $M$,
$\widehat{M}$ is cuspidal.
\end{proposition}

\begin{proof}
It suffices to work on one support coset, so we assume that
\(\operatorname{supp}M\subseteq\lambda+\Z^n\).
We need to show that the weight spaces of $\widehat{M}$
are uniformly bounded.
First we can see that the weight space
\(\widehat M_\lambda\) is spanned by
\[
\{\,t^\bm d_{j}\boxtimes M_{\lambda-\bm}
\mid 1\leq j\leq n,\ \bm\in\Z^n\,\}.
\]
Indeed, this follows directly from the definition of \(K(M)\), by
induction on the total \(x\)-degree of a representative of a
weight vector in \(\mg_n\otimes M\).

 Let
\[
W_n=\bigoplus_{i=1}^n
\C[t_1^{\pm1},\ldots,t_n^{\pm1}]d_i
\cong\operatorname{Der}
\C[t_1^{\pm1},\ldots,t_n^{\pm1}]
\]
and let \(V=\bigoplus_{\ba\in\Z^n}M_{\lambda+\ba}\). Then \(V\) is a
uniformly bounded weight \(W_n\)-module. The displayed spanning formula
identifies the sum \(\oplus_{\bm\in \Z^n}\widehat M_{\lambda+\bm}\) of  weight spaces of \(\widehat M\) with
the \(A\)-cover \(W_n\boxtimes V\) of the $W_n$-module $V$. By
\cite[Theorem~4.8]{BF1}, this cover is uniformly bounded. Hence
\(\widehat M\) is cuspidal.
\end{proof}

\begin{lemma}\label{lem:finite-length-cover}
Let \(X\) be a cuspidal \((A_n,\mg_n)\)-module whose support is
contained in one coset of \(\Z^n\). Then \(X\) has finite length as an
\((A_n,\mg_n)\)-module.
\end{lemma}

\begin{proof}
Fix \(\lambda\in\operatorname{supp}X\). Every nonzero subquotient
\(Y\) of \(X\) contains a nonzero weight vector. Since the
operators \(t^\ba\), \(\ba\in\Z^n\), act invertibly and shift weights
by \(\ba\), we have  \(Y_\lambda\neq0\). Consequently, each composition
factor in a composition series of \((A_n,\mg_n)\)-submodules of $X$ contributes a nonzero
subquotient to the finite-dimensional space \(X_\lambda\). Thus every
composition series of $X$ has length at most \(\dim X_\lambda\), and hence \(X\) has a finite length
composition series.
\end{proof}

\begin{theorem}\label{cuspidal}
If \(M\) is a simple cuspidal \(\mg_n\)-module, then \(M\) is
isomorphic to a simple quotient of \(T(A_n(\lambda),V)\) for some
\(\lambda\in\C^n\) and some finite-dimensional simple
\(\gl_n\)-module \(V\).
\end{theorem}

\begin{proof}
The assertion is immediate when \(M\) is trivial. Assume that \(M\) is
nontrivial. Since \(M\) is simple, \(M=\mg_nM\). Its support is
contained in one coset of \(\Z^n\). By Proposition~\ref{M-hat} and
Lemma~\ref{lem:finite-length-cover}, \(\widehat M\) is a cuspidal
\((A_n,\mg_n)\)-module of finite length. Let
\[
0=\widehat M_0\subseteq\widehat M_1\subseteq\cdots
\subseteq\widehat M_s=\widehat M,\qquad s\in\Z_{\geq1},
\]
be a composition series. Each \(\hat\theta(\widehat M_i)\) is a
\(\mg_n\)-submodule of \(M\). Hence there exists
\(0\leq j\leq s-1\) such that
\(\hat\theta(\widehat M_j)=0\) and
\(\hat\theta(\widehat M_{j+1})=M\). Thus \(M\) is a simple quotient of
the simple cuspidal \((A_n,\mg_n)\)-module
\(\widehat M_{j+1}/\widehat M_j\). By
Proposition~\ref{(A_n,g_n)-module},
\[
\widehat M_{j+1}/\widehat M_j\cong T(A_n(\lambda),V)
\]
for some \(\lambda\in\C^n\) and some finite-dimensional simple
\(\gl_n\)-module \(V\). The result follows.
\end{proof}

Note that  $T(A_n(\lambda),V)\cong T(A_n(\lambda+\bb),V)$ for any $\bb\in \Z^n$. Hence when $\lambda\in \Z^n$, we can assume that $\lambda=\bo$. By Theorem 2, Proposition 3 in \cite{Z}, Theorems~\ref{simplicity}
and~\ref{remain}, simple subquotients of $T(A_n(\lambda),V)$ can be described as follows.

Let $\lambda\in \C^n$ and \(V\) a simple
\(\gl_n\)-module. Then:
\begin{enumerate}
\item \(T(A_n(\lambda),V)\) is simple if
\(V\ncong V(\delta_r,r)\) for \(0\leq r\leq n\).

\item If $\lambda\not\in \Z^n$, then the complex (\ref{complex}) is exact.
\item For $1\leq r\leq n-1$, $\tilde{L}_n(A_n,r)=L_n(A_n,r)+\C t^{\bo}x^{\bo}\otimes V(\delta_r,r)$. And $T(A_n,(\delta_n,n))$ is simple.  $T(A_n,(\delta_0,0))=A_n$ has the only simple submodule $\C t^{\bo}x^{\bo}$.
\end{enumerate}

\vspace{2mm}

\noindent

\vspace{4mm}

\noindent G. L.: School of Mathematics and Statistics,
Henan University, Kaifeng 475004, China. Email: liugenqiang@henu.edu.cn

\noindent  X. Z.: School of Mathematics and Statistics,
Henan University, Kaifeng 475004, China. Email: XYZheng@henu.edu.cn

\noindent  Y. Z.: School of Mathematical Sciences, Hebei Normal University, Shijiazhuang 050024, China. E-mail: zhaoyf@hebtu.edu.cn

\vspace{0.2cm}

\end{document}